\documentclass[10pt]{article}
\usepackage{amsmath,amssymb,stmaryrd,amsthm}
\usepackage[pdftex]{graphicx}
\usepackage{fullpage}
\usepackage{cite}
\usepackage{palatino}
\usepackage{hyperref}

\usepackage{tikz}\usetikzlibrary{cd}
\usetikzlibrary{arrows}

\usepackage[framemethod=TikZ]{mdframed}
\mdfdefinestyle{MyFrame}{%
    linecolor=orange,
    outerlinewidth=1pt,
    roundcorner=3pt,
    innertopmargin=.2\baselineskip,
    innerbottommargin=.2\baselineskip,
    innerrightmargin=20pt,
    innerleftmargin=20pt,
    backgroundcolor=blue!80!white}

\newcommand{\Z}{\mathbb{Z}}

\newtheorem{thm}{Theorem}[section]
\newtheorem{lem}[thm]{Lemma}
\newtheorem{prop}[thm]{Proposition}
\newtheorem{defn}[thm]{Definition}
\newtheorem{definition}[thm]{Definition}
\newtheorem{remark}[thm]{Remark}
\newtheorem{rem}[thm]{Remark}
\newtheorem{cor}[thm]{Corollary}
\newtheorem{exa}[thm]{Example}

\numberwithin{thm}{subsection}

\newcommand{\av}[1]{\left|{#1}\right|}

\newcommand{\be}{\begin{enumerate}}
\newcommand{\bi}{\begin{itemize}}
\newcommand{\ee}{\end{enumerate}}
\newcommand{\ei}{\end{itemize}}
\newcommand{\ii}{\item}

\newcommand{\zm}{\Z/m\Z}
\newcommand{\zmod}[1]{\Z/{#1}\Z}
\newcommand{\ord}{\mathsf{ord}}
\newcommand{\ordt}[1]{\mathsf{ord}_{#1}(2)}
\newcommand{\add}[1]{\mathcal{A}_{{#1}}}

\newcommand{\gr}[1]{\mathcal{G}({#1})}
\newcommand{\units}[1]{\mathcal{U}({#1})}
\newcommand{\nilp}[1]{\mathcal{N}({#1})}

\DeclareMathOperator{\lcm}{lcm}

\newcommand{\tx}{\widetilde{x}}
\newcommand{\ty}{\widetilde{y}}
\newcommand{\hx}{\widehat{x}}
\newcommand{\hy}{\widehat{y}}
\newcommand{\tree}[3]{{\mathrm{Tree}}_{#1}^{({#3})}({#2})}
\newcommand{\kgp}{\oblong}
\newcommand{\kpg}{\kgp}

\renewcommand{\phi}{\varphi}

\title{The shape of $x^2\bmod n$}
\author{Lee DeVille\\Department of Mathematics, University of Illinois}

\begin{document}

\maketitle

\begin{abstract}
  We examine the graphs generated by the map $x\mapsto x^2\bmod n$ for various $n$, present some results on the structure of these graphs, and compute some very cool examples.
\end{abstract}


\section{Overview}

For any $n$, we consider the map 
\begin{equation*}
  f_n(x) := x^2\bmod n.
\end{equation*}
From this map, we can generate a (directed) graph $\gr n$ whose vertices are the set $\{0,1,\dots, n-1\}$ and edges from $x$ to $f_n(x)$ for each $x$.  We show a few examples of such graphs  later in the paper --- one particularly intruiging graph is shown in Figure~\ref{fig:1024attractor}, that appears as a subgraph of $\gr{10455}$.  Figure~\ref{fig:817}, which shows the entire graph $\gr{817}$, is also excellent. As we can see from examples, sometimes the graph $\gr n$ is a forest (a collection of trees), sometimes it has loops, etc.  In this paper we will seek to characterize all such graphs; not surprisingly, the shape of these graphs depend on the number theoretic properties of the integer $n$.

We give a quick summary of the results below:

\be

\ii  (Theorem~\ref{thm:kronecker}) We show that if $n=a\cdot b$ where $\gcd(a,b)=1$, then the graph $\gr n$ is the Kronecker product of the graphs $\gr a$ and $\gr b$ (we define the Kronecker product in Definition~\ref{def:kronecker} below).  Using induction,  we can therefore determine the graph $\gr n$ in terms of its prime decomposition, and all that remains is to describe $\gr{p^k}$ for all primes $p$ and $k\ge 1$.   We break this up into several cases.

\ii (Theorem~\ref{thm:units})   Let $p$ be an odd prime.  Let us denote by $\units{p^k}$ the set of numbers that are relatively prime to $p^k$; these are the multiplicative units modulo $p^k$.   We use heavily the fact that the units under multiplication modulo $p^k$ form a cyclic group, and this allows us to characterize $\units{p^k}$ (we note here that the technique for $\units{p^k}$ follows closely the results of~\cite{vasiga2004iteration} where they studied $\gr p$ for $p$ prime);

\ii (Theorem~\ref{thm:nilpotent}) Let $p$ be an odd prime.  This theorem will fully characterize $\nilp{p^k}$, the component of $\gr{p^k}$ that corresponds to the nilpotent elements.

\ii (Section~\ref{sec:2^k}) Finally, we work on the graph for $\gr{2^k}$, describing both $\units{2^k}$ and $\nilp{2^k}$ separately. 

\ee

From this, we can  characterize any $\gr n$, and we then use this characterization to compute several quantities of interest and work out many sweet examples.

A few notes:  this manuscript does not contain any new theoretical results and, in fact, only uses results from elementary number theory\footnote{Anyone familiar with the author's body of research would know that any number theory appearing here is perforce extremely elementary}.  If the reader is so inclined, they can think of this paper as an extended pedagogical example.  In fact, the author recently taught UIUC's MATH 347 (our undergrad math major "intro to proofs" course) and found that the students responded pretty well to visualizing the graphs of various functions that appear in a math course at this level (functions modulo some integer being a key family of examples).  The author also made some videos with visualizations for the square function considered here~\cite{vid}, and got intrigued by the patterns that appear.  This led to the current manuscript.

\section{The main theory}

\setcounter{subsection}{-1}
\subsection{Summary of results}

As noted above, there are four main results, and the sections below address each of these in order.   After we have stated the theoretical results in this section, this is sufficient to describe all $\gr n$ --- with the caveat that in many cases there is still some work to do to compute everything.  We work out many concrete examples in Section~\ref{sec:computations}.

\subsection{Kronecker products and the Classical Remainder Theorem}\label{sec:Kronecker}

In this section we show how the graphs $\gr a$ and $\gr b$ ``combine'' to form the graph of $\gr{ab}$, when $a,b$ are relatively prime.

\begin{definition}\label{def:kronecker}
  Let $G_1 = (V_1,E_1)$ and $G_2 = (V_2,E_2)$ be two (directed) graphs.  Let us  define $G_1\kgp G_2$, the {\bf Kronecker product} of $G_1$ and $G_2$, as follows:  the vertex set of $G_1\kgp G_2$ is the Cartesian product $V_1\times V_2$, and we say that 
  \begin{equation*}
    (v_1,w_1)\to (v_2,w_2)\in G_1\kgp G_2\mbox{  iff  } v_1\to w_1\in G_1\mbox{ and } v_2\to w_2\in G_2.
  \end{equation*}
  In particular, the edge exists for the pair iff there is an edge for the first component of each and one for the second component of each.
\end{definition}

\begin{remark}
  One motivation for calling this the Kronecker product is that the adjacency matrix of $G_1\kgp G_2$ is the Kronecker matrix product of the adjacency matrices of $G_1$ and $G_2$.  This is also called the Cartesian product of graphs by many authors.
    We also note that it is straightforward to show that the Kronecker graph product is associative, and so we can define $n$-ary Kronecker products unambiguously as well.
\end{remark}

\begin{thm}\label{thm:kronecker}If $gcd(a,b)=1$, then 
\begin{equation*}
  \gr{ab} \cong \gr a \kgp \gr b.
\end{equation*}
Furthermore, If we write
\begin{equation*}
  n = p_1^{\alpha_1} p_2^{\alpha_2} \cdots p_j^{\alpha_j} 
\end{equation*}
as the prime factorization of $n$, then
\begin{equation*}
  \gr n \cong \bigbox_{i=1}^j \gr{p_i^{\alpha_i}},
\end{equation*}
namely:  the graph for $n$ is just the Kronecker product of the individual graphs for each of the $p_i^{\alpha_i}$.\end{thm}

Perhaps surprisingly, the proof of Theorem~\ref{thm:kronecker} is more or less equivalent to the Classical Remainder Theorem (see Lemma~\ref{lem:crt} below).  The basic idea goes as follows:  if we consider the map $x\mapsto x^2\bmod {ab}$, then we can naturally map this to a pair where we first consider the operation modulo $a$, and then consider the operation modulo $b$.  (In fact, we can do this for any $a,b$, not just those that are relatively prime.)  However, if $a,b$ are relatively prime, then the CRT tells us that we can invert this process, and this is what gives us the full isomorphism.

\begin{lem}\label{lem:crt}[Classical Remainder Theorem (CRT)]~\cite[\S 7.6]{dummit1991abstract}
  Given a set of coprime numbers $n_1, n_2, \dots, n_j$, then for any integers $a_1,\dots, a_j$, the system
  \begin{equation*}
    x \equiv a_1\bmod n_1,\quad    x \equiv a_2\bmod n_2,\quad\cdots, \quad  x \equiv a_j\bmod n_j 
  \end{equation*}
  has a solution.  Moreover, any two such solutions are congruent modulo $n_1\times n_2\times \cdots \times n_j$.
\end{lem}

\begin{remark}\label{rem:CRT}
  An equivalent statement of the CRT is that the map
  \begin{align*}
     \zmod{(n_1\times n_2\times \cdots \times n_j)} &\to \zmod{n_1}\times \zmod{n_2}\times\cdots\times \zmod{n_j},\\
     x\bmod{(n_1\times n_2\times \cdots \times n_j)} &\mapsto (x\bmod n_1, x\bmod n_2, \dots, x\bmod n_j), 
  \end{align*}
  is a ring isomorphism, which is why it respects the squaring operation.
\end{remark}

\begin{proof}[Proof of Theorem~\ref{thm:kronecker}]
  The first claim basically boils down to the interpretation of Remark~\ref{rem:CRT}.  Let us consider the map from the vertex set of $\gr{ab}$ to the vertex set of  $\gr a\kgp \gr b$ given by $x\bmod(ab) \mapsto (x\bmod a, x\bmod b)$.  By the CRT, this is a bijection of the vertex sets.  Now, assume that there is an edge $x\to y$ in $\gr{ab}$ --- this will be true iff $x^2=y\bmod{(ab)}$.  But then $x^2=y\bmod a$ and $x^2=y\bmod b$, and therefore there is an edge $(x\bmod a, x\bmod b) \to (y\bmod a, y\bmod b)$ in $\gr a\kgp \gr b$ as well.  Conversely, if there is an edge $(x\bmod a, x\bmod b) \to (y\bmod a, y\bmod b)$ in $\gr a\kgp \gr b$, then by CRT we have $x^2=y\bmod{(ab)}$, and there is an edge $x\to y$ in $\gr{ab}$.  
  Finally, the second claim follows directly from the first using induction.
\end{proof}

\subsection{The graph of units $\units{p^k}$ when $p$ is an odd prime}\label{sec:units}

In this section, we determine how to compute the graph containing the ``units'' modulo $p^k$ when $p$ is an odd prime.  We start with a few definitions and state the main result, and then break down how we prove it.

\begin{defn}
  The set of {\bf multiplicative units} modulo $n$ (or, more simply, the {\bf units} modulo $n$) are those numbers in the set $\{0,1,\dots, n-1\}$ that are relatively prime to $n$.  This set is typically denoted $(\zmod{n})^\times$ and the function $\phi(n) = \av{(\zmod{n})^\times}$ is called Euler's phi function (or totient function).  In this paper we will denote by $\units{p^k}$ the graph induced by $x\mapsto x^2\bmod{p^k}$ on the set of units.
\end{defn}

\begin{defn}
  Let $\gcd(d,k)=1$.  We define the {\bf multiplicative order of $k$, modulo $d$}, denoted $\ord_d(k)$, as the smallest power $p$ such that $k^p\equiv 1\bmod d$, or $d|(k^p-1)$. 
\end{defn}

\begin{defn}\label{def:flower-cycle}
Here we define several special graphs that appear all over the place in $\units{p^k}$:

\bi

\ii The {\bf (directed) cycle graph} of length $k$, denoted $C_k$, is a directed graph with vertex set $\{1,\dots, k\}$ and arrows $v_i\to v_{i+1\bmod k}$.  Note that we also allow the cycle of length 2, given by the graph $1\to 2, 2\to 1$, and the directed cycle of length $1$, which is just a single vertex labelled 1 with a loop to itself.

\ii 
A {\bf tree} is a connected graph with no cycles. A {\bf rooted tree} is a tree with a distinguished vertex, called the root.    A {\bf grounded tree} is a rooted tree with two properties:  all edges flow towards the root (i.e. for any vertex in the tree, there is a path from that vertex to the root), and the root has a loop to itself. 

\ii The {\bf flower cycle of length $\alpha$ and flower type $T$}, denoted $C_\alpha(T)$, is the directed graph defined in the following manner.  Take $k$ copies of the grounded tree $T$, denoted $T_1,\dots, T_k$.  Remove the loop at the root from each of these, and then replace it with an edge from the root of tree $i$ to the root of tree $i+1\bmod \alpha$.  

\ii 
The {\bf regular grounded tree of width $w>1$ and $\ell>0$ layers} (denoted $T_{w}^{\ell}$) is a tree with $w^\ell$ vertices, described as follows:  start with a root vertex and a loop to itself, this forms layer 0.  In layer 1, there are $w-1$ vertices, each with an edge to the root.  For any layer $k< \ell$, and for any vertex in layer $k$, there are $w$ vertices in layer $k+1$ that map to each vertex in layer $k$.   Layer $\ell$ is given by leaves that map to vertices in level $\ell-1$.
\ei
\end{defn}

\begin{remark}
A few remarks:
\bi\ii  Another way to describe the flower cycle is that we start with a cycle, and then we glue $k$ copies of the tree at each vertex in the cycle.  

\ii Note that the in-degree of every vertex in $T_w^\ell$ is either $w$ or 0.

\ii The graph $C_\alpha(T_w^\theta)$ appears often in the sequel (esp. with $w=2$), and we give a colloquial description here:  start with a cycle of length $C_\alpha$, and then to each vertex in the cycle, ``glue'' a tree of type $T_w^\theta$. In this picture, the nodes in the cycle itself are the periodic points, and the trees coming off each node are the preperiodic points that map into that particular orbit. See Figure~\ref{fig:flower-examples} for a few examples.

\ei

\end{remark}

\begin{figure}[th]
\begin{center}
\end{center}
\includegraphics[width=0.95\textwidth]{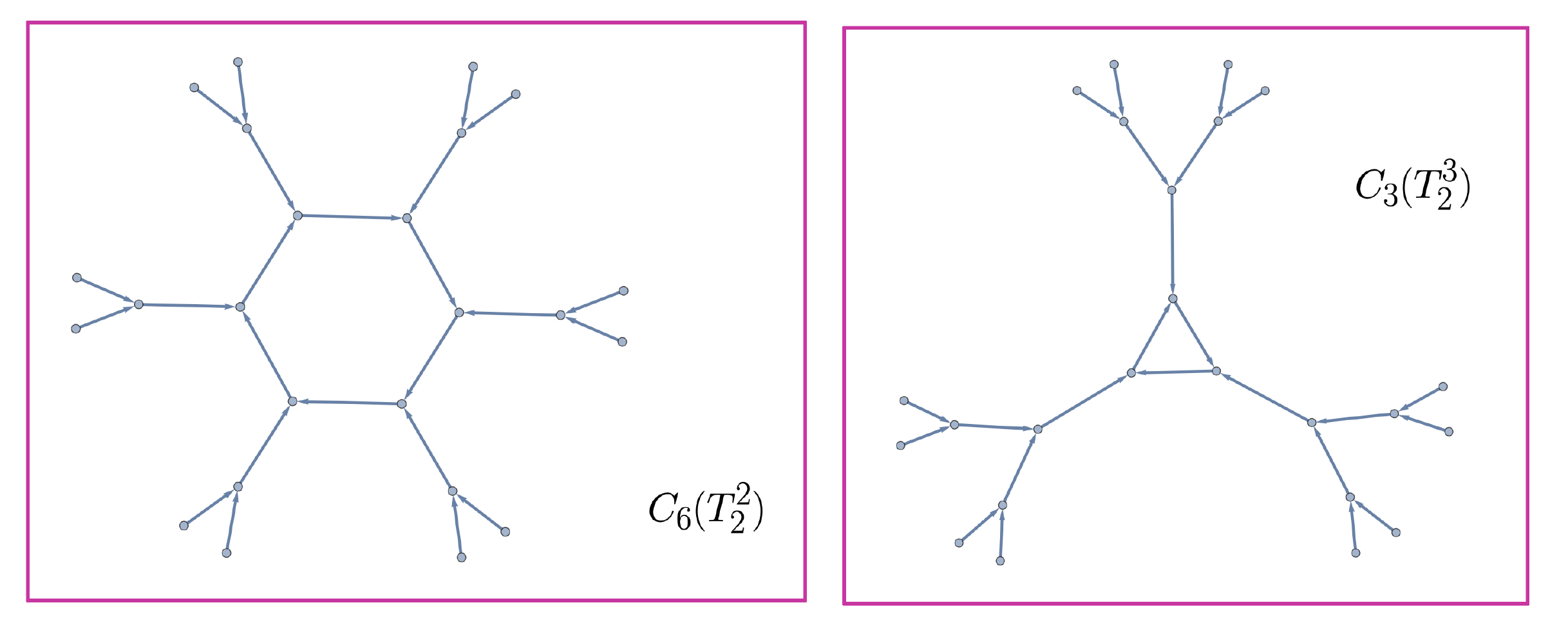}
\caption{Examples of $C_6(T_2^2)$ and $C_3(T_2^3)$ as defined in Definition~\ref{def:flower-cycle}.  Recall that the subscript gives the number of terms in the cycle (or periodic orbit) of the graph, and the tree tells us what is ``hanging off'' of each of those periodic vertices.  In one case we have the grounded tree of width two and depth two, and in the other width two and depth three.  Trees of width two show up ubiquitously below --- in particular they are the only trees we see in $\units{p^k}$ when $p$ is an odd prime.}
\label{fig:flower-examples}
\end{figure}

\begin{thm}\label{thm:units}[Structure of $\units{p^k}$]
Let $p$ be an odd prime, and write $m=\phi(p^k)$.  Write $m = 2^\theta\mu$ where $\mu$ is odd.  (Note that $\phi(p^k)$ is always even, so $\theta\ge 1$.)  For each divisor $d$ of $\mu$, take $\phi(d)/\ordt d$ disjoint copies of $C_{\ordt d}(T^\theta_2)$, and $\units{p^k}$ is isomorphic to the disjoint union of these cycles over the divisors $d$. 
\end{thm}

\begin{remark}\label{rem:2or0}
  Note that this implies that every vertex in $\units{p^k}$ has an in-degree of two or zero (note that the trees attached to the cycles are always of the form $T_2^\theta$).
\end{remark}

To describe the result a bit more colloquially:  we first compute $m = \phi(p^k)$, and pull out all possible powers of $2$, until we obtain the odd number $\mu$.  From there, we enumerate all of the divisors $d$ of $\mu$, and these $d$ determine all of the periodic orbits of $\units{p^k}$.  Then the number of powers of $2$ that we pulled out gives us the full structure of the pre-periodic points attached to these periodic orbits.  


Onto the proof:  we start with the following result that  was proven by Gauss in~\cite{gauss1986english}:
\begin{prop}
  If $n=p^k$, where $p$ is an odd prime, then $(\zmod{p^k})^\times$ is a cyclic group of order $\phi(p^k) = p^k-p^{k-1}$.
\end{prop}
The fact that this group is cyclic allows us to more or less completely characterize the graph structure of $\units{p^k}$.  We follow the approach of~\cite{vasiga2004iteration} who completely characterized $\gr p$ for $p$ prime in the same way (in this section the novelty is that we are extending these ideas to $p^k$, but since $(\zmod{p^k})^\times$ is cyclic the same idea goes through).  

We first give a definition and a lemma:

\begin{defn}
  We define $\add{m}$ as the graph whose vertices are the set $\{0,1,\dots, m-1\}$ and edges given by $x\mapsto 2x\bmod m$.
\end{defn}

\begin{lem}\label{lem:cyclic}
  Let $G$ be any cyclic group of order $m$ (where here we denote the group operation as multiplication).  Call one of its generators $g$.  We can define a function on $G$ by 
  \begin{equation*}
    g^\alpha \mapsto g^{2\alpha},\quad \alpha = 0,1,\dots, m-1.
  \end{equation*}
  This function is conjugate to the map $(\zm, x\mapsto 2x\bmod m)$, and so has the same dynamical structure (periodic orbits, fixed points, etc.) and its graph is isomorphic to $\add m$.
\end{lem}

\begin{proof}
  Let us define the maps $\alpha\colon \zm \to \zm$ with $x\mapsto 2x\bmod m$,  $\gamma\colon G\to G$ with $\gamma(g^a) = g^{2a}$, and $\kappa\colon\zm\to G$ with $\kappa(a)  = g^a$.  Since $g$ generates $G$, the map $\kappa$ is invertible.  We see that $\gamma\circ \kappa = \kappa\circ \alpha$, so $\alpha$ and $\gamma$ are conjugate. In particular, note that $\alpha(x)=x$ iff $\gamma(\kappa(x)) = \kappa(x)$ and moreover that $\alpha^k (x) = x$ iff $\gamma^k(\kappa(x)) = \kappa(x)$, so that there is a one-to-one correspondence between the fixed/periodic points of $\alpha$ and $\gamma$.  In particular this implies that the two functions have isomorphic graphs.
  \end{proof}

\begin{proof}[Proof of Theorem~\ref{thm:units}]
  By Lemma~\ref{lem:cyclic}, the graph $\units{p^k}$ is isomorphic to $\add{\phi(p^k)}$ and the result follows.
\end{proof}

Basically we just need to understand $\add m$.  Under addition, this is a cyclic group of order $m$, with generator $1$.  We stress here that we are now moving to {\em additive notation} for simplicity, so here the identity is $0$ and the generator is $1$; when we convert this back to the multiplicative group $\units{p^k}$ we will have an identity of $1$ and a generator of $g$.
 Some standard results from group theory~\cite[\S 2.3]{dummit1991abstract} tell us that: 
\be
\ii For any $d$ dividing $m$, the set of all elements with (additive) order $d$ is given by $$S_d = \{j (m/d): j\mbox{ such that }\gcd(j,d)=1\};$$
\ii The set $S_d$ has $\phi(d)$ elements;
\ii Since $\sum_{d|m}\phi(d) = m$, these sets exhaust $\zm$.
\ee

Let us first consider the case with $m$ odd:

\begin{prop}
  Let $m$ be odd, and consider the map $f(x)=2x$ on $\zm$.   Then this map has a unique fixed point at $x=0$, and all other elements are periodic under this map.  The map $f$ leaves each of the $S_d$ invariant, and  $S_d$ decomposes into a disjoint union of periodic orbits, each with period $\ord_d(2)$.  (Note that since $\av{S_d} = \phi(d)$, there will be exactly $\phi(d)/\ord_d(2)$ such disjoint orbits and note that $\ordt d$ divides $\phi(d)$ by Fermat's Little Theorem.)
\end{prop}

\begin{proof}
We can see by inspection that no point other than $x=0$ is fixed by the map.  Now note that since $\gcd(2,m)=1$, the map $x\mapsto 2x$ is invertible, and is thus effectively a permutation.  Therefore all points are periodic under repeated application of this map.

To prove the second claim in the proposition, note that $S_d$ is the set of points in $\zm$ with additive order $d$, i.e. $x\in S_d$ iff $dx\equiv0\bmod m$ and if $0<d'<d$, $d'x\not\equiv 0\pmod m$.  This is invariant under the map $x\mapsto 2x$ since $m$ is odd:  if $d'(2x)\equiv 0\bmod m$ then $d'x\equiv 0 \bmod m$, etc. 

Let $x\in S_d$ with $f^k(x) = 2^kx\equiv x$, and if $x\in S_d$ then $x=j(d/m)$ for $\gcd(j,d)=1$.  This means that we have 
\begin{equation*}
  (2^k-1) j\frac m d \equiv 0\bmod m,
\end{equation*}
which is equivalent to saying that $(2^k-1)j/d\in \Z$.  Since $\gcd(j,d)=1$, this is true iff $d|(2^k-1)$.  

The period of $x$ is the smallest $k$ for which $2^kx=x$, which is the smallest $k$ for which $d|(2^k-1)$, and this is $\ord_d(2)$ by definition.
\end{proof}

We are now able to understand the cyclic group for more general $m$.

\begin{prop}
Let $m = 2^\theta \mu$ where $\mu$ is odd, and again consider the map $f(x)=2x\bmod m$.  Then the graph $\add m$ consists of flower cycles, specifically:  replace every cycle of length $\alpha$ in $\add\mu$ with a copy of $C_\alpha(T_2^\theta)$.
\end{prop}

\begin{proof}
  We first consider the basin of attraction of $0$, which consists of the multiples of $\mu$.  (Note that $f^\theta(x)=2^\theta x\equiv 0\bmod m$ for some $\theta$ iff $x$ is a multiple of $\mu$.)  In fact, we can see that if $x= 2^q \nu \mu$ with $q<\theta$ and $\nu$ odd, then $x$ reaches 0 in exactly $k-q$ iterations, each passing through higher orders of $2$.  For example, all of the odd multiples of $\mu$ will be in the top layer (there will be exactly $2^{\theta-1}$ of these); the odd multiples of $2\mu$ will be in the next layer (exactly $2^{\theta-2}$ of these), and so on.  Note that every vertex in this tree has in-degree 2 or 0:  nothing maps to the odd numbers in the top layer, but each subsequent number has two preimages:  given an $x$ such that $2y\equiv x\bmod m$, we also have $2(y+m/2) \equiv x\bmod m.$  So in general, we have a tree of type $T_2^\theta$ with $0$ at the root, the point $m/2 = 2^{\theta-1}\mu$ the single point in layer $1$, the points $\{1,3\}\times 2^{\theta-2}\mu$ in layer two (each of which map to $2^{\theta-1}\mu$), the points $\{1,3,5,7\}\times 2^{\theta-3}\mu$ in layer three, etc.  We will call this the ``tree rooted at 0'' below.
  
We now claim that for any periodic orbit that appears in $\add \mu$, there is a corresponding flower orbit in $\add m$.  Consider any period $k$ orbit in $\zmod \mu$ with elements $x_1,\dots, x_k$.  By definition, $x_{i+1}\equiv 2x_i\bmod \mu$ and $x_1 = 2x_k\bmod \mu$.  Now let $y_i = 2^\theta x_i$ and note that gives a periodic orbit modulo $m$:  if $x_{i+1}-2x_i$ is a multiple of $\mu$, then $y_{i+1}-2y_i = 2^\theta(x_{i+1}-x_i)$ is a multiple of $m=2^\theta \mu$.  Now pick any of the $y_i$.  We know that there is at least one $z\in \zm$ that maps to $y_i$ (since it is part of a periodic orbit).  Therefore there is another, $z_1 = z\pm n/2$ that maps to $y_i$, and this corresponds to the $n/2$ vertex in the tree rooted at 0.  If $z$ is odd, then the tree terminates, but if it is even, there are two numbers that map to $z_1$, namely $z_1/2$ and $z_1/2\pm n/2$.  As we move up the branches of this tree, each layer removes one power of two, and each vertex has exactly two vertices mapping to it, until we reach the $\theta$'th layer.  Therefore there will be a $T_2^\theta$ attached to $y_i$ --- and the same argument applies to any of the elements in any of the periodic orbits, and we are done.

\end{proof}

\begin{cor}\label{cor:2^k}
  If $m=2^k$ for some power of $k$, then the graph $\add{m}$ is a tree with $k$ layers.  The $k$th layer consists of the $2^{k-1}$ odd numbers, the $k-1$st layer is those  $2^{k-2}$ numbers that are two times an odd, etc.  In particular, the graph $\add{2^k}$ is a $T_2^k$.
\end{cor}

One direct corollary of this is that if $n$ is a Fermat prime, then the graph $\units{n}$ is a tree.

\subsection{The graph of nilpotent elements $\nilp{p^k}$ when $p$ is an odd prime}\label{sec:nilpotent}

Now we can attack the case of those elements that are not units when $p$ is an odd prime.

\begin{defn}
  We say that $x$ is {\bf nilpotent} modulo $n$ if there exists some power $a$ such that $x^a \equiv \bmod n$.  
\end{defn}

It is easy to see that $x$ is nilponent modulo $p^k$ iff $p|x$.  Also note that if $x^a\equiv 0\bmod n$ then $x^b\equiv 0\bmod n$ for any $b>a$, and in particular it is true for $b$ being some power of two, so we have: 

\begin{lem} 
$x$ is nilpotent modulo $p^k$ $\iff$ $p|x$ $\iff$ $f_{p^k}^\alpha(x) = 0$ for some $\alpha>0$.  In particular, this implies that $\nilp{p^k}$ is a (connected) tree.
\end{lem}

Now that we have established that $\nilp{p^k}$ is a tree and we want to determine its structure.  It can actually have an interesting and complex structure, especially when $k$ is large.  Our main result is the following.

\begin{defn}\label{def:trees}\label{def:tree}
  Let $p$ be an odd prime.  Fix $\hx \in \units p = \{1,2,\dots,p-1\}$, and for each $\ell$ we define the tree $\tree p \hx \ell$ recursively:
  \be
  
  \ii If $\ell$ is odd: then $\tree p \hx\ell$ is just a single node with no edges;
  
  \ii If $\ell$ is even and $\hx$ has no preimages modulo $p$: then  then $\tree p \hx\ell$ is just a single node with no edges;
  
  \ii If $\ell$ is even and $\hx$ has two preimages modulo $p$: denote these preimages by $z_1$ and $z_2$, then $\tree p \hx\ell$ is a rooted tree with $2 p^{\ell/2}$ trees mapping into this root:
  \begin{equation*}
    \mbox{ $p^{\ell/2}$ copies of $\tree p {z_1}{\ell/2}$ and $p^{\ell/2}$ copies of $\tree p {z_2}{\ell/2}$.}
  \end{equation*}
  \ee
  (Note by Remark~\ref{rem:2or0} that this exhausts all possible cases.)
\end{defn}

\begin{thm}\label{thm:nilpotent}
  The graph $\nilp{p^k}$ is a rooted tree with $0$ as the root and $p^{k-1}$ vertices.  The in-neighborhood of 0 is constructed as follows:  for any $k/2\le \ell < k$ and for any $1 \le y \le p^{k-\ell}$ with $\gcd(y,p)=1$, attach a tree of type $\tree{p}{\hy}{\ell}$, where $\hy=y\bmod p$, to 0.  For any $x$ (where $x\bmod p = 0$ and $x\bmod{p^k}\neq0$), write $x$ as the (unique) decomposition $x = \tx p^\ell$ with $\gcd(\tx,p)=1$, and the in-neighborhood of $x$ is the tree  $\tree{p}{\hx}{\ell}$, where $\hx=\tx\bmod p$. 
\end{thm}

\begin{remark}
  The descriptions given in Theorem~\ref{thm:nilpotent} and Definition~\ref{def:trees} are recursive and not explicit --- although it is relatively easy to use these descriptions to compute $\nilp{p^k}$ when $k$ is not too large.  We give an equivalent, but more explicit, description below.
\end{remark}

Theorem~\ref{thm:nilpotent} basically follows from the following lemma:

\begin{lem}\label{lem:q2}
  Consider $n=p^k$, let $x=\tx p^\ell$ (where $\ell < k$) and consider the equation
\begin{equation}\label{eq:q2}
  q^2 = x\bmod{p^k}.
\end{equation} 
Write $\hx$ as the element of $\{1,2,\dots, p-1\}$ that is congruent to $\tx$ modulo $p$.  If $\ell$ is odd, or if $\ell$ is even and $\hx$ is a leaf in $\units p$, then~\eqref{eq:q2} has no solutions.  If there are two numbers $\widetilde z_1,\widetilde z_2$ such that $\widetilde z_i^2 = \hx\bmod p$, then~\eqref{eq:q2} has $2p^{\ell/2}$ solutions.  Half of these ($p^{\ell/2}$ of them) are solutions of the form $z p^{\ell/2}$ where $z\equiv z_1\bmod p$ and the other half are solutions of the form $z p^{\ell/2}$ where $z\equiv z_2\bmod p$.
\end{lem}

\begin{proof}
  Let us first consider the case where $\ell = 2$ as a warm-up; the general case is similar but has more details.
  
  \fbox{$\ell=2$.}  Let us write $x = \tx p^2$, $y=\ty p$, and assume that $y^2 = x\bmod {p^k}$.  This tells us that 
  \begin{align*}
    \ty^2 p^2 &\equiv \tx p^2 \bmod {p^k},\\
    \ty^2 &\equiv \tx \bmod{p^{k-2}}.
  \end{align*}
  Note that we have $1\le \tx < p^{k-2}$ and $1\le \ty < p^{k-1}$ --- this is the range of prefactors we can put in front of each of those powers.
   
  Let us write $\hx = \tx\pmod p$ and $\hy = \ty\pmod p$.  If we consider the last equation modulo $p$, this implies $\hy^2\equiv \hx\bmod p$.  This implies that if $\hx$ has no preimages in $\gr p$ then there is no solution.  Let us assume that $\hx$ has two preimages in $\gr p$, and pick one of them, say $z_1$.  If we can show that there are exactly $p$ solutions to the system of congruences
  \begin{equation*}
    \ty^2 \equiv \tx \bmod{p^{k-2}},\quad \ty\equiv z_1\bmod p,
  \end{equation*}
  then we have established the result of the theorem for $\ell=2$.  Let us define two sets:
  \begin{equation*}
    S_x := \{\tx: 1 \le \tx < p^{k-2} \land \tx \equiv \hx \bmod p\},\quad     S_y := \{\ty: 1 \le \ty < p^{k-1} \land \ty \equiv \hy \bmod p\}.
  \end{equation*}
  We can see that $\av{S_x} = p^{k-3}$ and $\av{S_y} = p^{k-2}$, and of course $pS_y$ maps into $p^2S_x$ under the squaring map.  Thus each point in $S_x$ is hit an average of $p$ times, but we want to show that it is exactly $p$ times.  Let us parameterize $S_y$ by $\ty = z_1 + \beta p$, with $0\le \beta < p^{k-2}$.    We claim that as we run through these $\beta$, the values of $S_x$ each get hit exactly $p$ times.  First note that if we replace $\beta$ with $\beta+p^{k-3}$, then we get
  \begin{equation*}
    (z_1 + (\beta +p^{k-3})p)^2 \equiv (z_1+\beta p)^2\bmod p^{k-2},
  \end{equation*}
  so we only need consider $0\le \beta < p^{k-3}$.  We claim that these $\beta$ give distinct values modulo $p^{k-2}$, and since there are $p^{k-3}$ of them, they must cover the set $S_x$ --- and then shifting the arguments by $p^{k-3}$ gives the same values, and this gives us $p$ copies.  All we need to show now is that every element of $S_x$ gets hit once as $\beta$ runs from $0$ to $p^{k-3}-1$.  We compute:  
  \begin{align*}
    (1+\beta p)^2 &\equiv (1+\beta' p)^2&\pmod {p^{k-2}}\\
    1+2\beta p + \beta^2p^2 &\equiv 1+2\beta' p + (\beta')^2 p^2&\pmod {p^{k-2}}\\
    2\beta + \beta^2p &\equiv 2\beta' + (\beta')^2 p&\pmod {p^{k-3}},\\    
    2(\beta-\beta') + p(\beta^2-(\beta')^2)&\equiv 0&\pmod {p^{k-3}},
  \end{align*}
  so we nowneed to show that the map $2\xi + p\xi^2$ has only one root on $\zmod{p^{k-3}}$.  If we assume that $2\xi+p\xi^2\equiv 0\mod{p^{k-3}}$, then by taking modulo $p$, this gives $2\xi \equiv 0\bmod p$, and thus $\xi\equiv 0\bmod p$.  From this we have $p\xi^2\equiv 0\bmod {p^3}$, thus the same for $2\xi$, thus for $\xi$, etc.  From this we get that only $\xi=0$ can solve this equation, and we are done.
  
  \fbox{General $\ell$.}
  The argument is similar:  writing  $x = \tx p^\ell$, $y=\ty p^{\ell/2}$, and that $y^2 = x\bmod {p^k}$.  This tells us that 
  \begin{align*}
    \ty^2 p^\ell &\equiv \tx p^{\ell} \bmod {p^k},\\
    \ty^2 &\equiv \tx \bmod{p^{k-\ell}}.
  \end{align*}
  Note that we now have $1\le \tx < p^{k-\ell}$ and $1\le \ty < p^{k-\ell/2}$.   Again assume that $\hx$ has preimages in $\gr p$ and call one $z_1$. We now show that there are $p^{\ell/2}$ solutions to the system 
   \begin{equation*}
    \ty^2 \equiv \tx \bmod{p^{k-\ell}},\quad \ty\equiv z_1\bmod p.
  \end{equation*}
  Here we parameterize $\ty = z_1 + \beta p$ with $0\le \beta < p^{k-\ell/2-1}$.  We obtain the same repeating argument when adding a multiple of $p^{k-\ell-1}$, and so we only need to consider $0\le \beta < p^{k-\ell-1}$. But this will repeat $p^{(k-\ell/2-1)-(k-\ell-1)} = p^{\ell/2}$ times, so again we only need to show that every target gets hit in the cycle. But then the rest of the argument follows.

\end{proof}

\begin{proof}[Proof of Theorem~\ref{thm:nilpotent}]
  The lemma does most of the work for us here. Note that implicit in the lemma that for any $x$, if we write $x = \tx p^\ell$, then the basin of attraction of $x$ depends only on $\hx = \tx \bmod p$.  We can go through the cases:  if $\ell$ is odd, then $x$ has no preimages modulo $p^k$ by the lemma, and so it is a leaf in $\nilp{p^k}$ --- and similarly for $\ell$ even but $\hx$ having no preimages modulo $p$.
  
  Finally, for the neighborhood of 0:  note that if $\ell \ge k/2$, then $x^2\equiv 0 \bmod p^k$.  Thus any such point maps to 0 under squaring, and the preimage of each of these points is a tree of type $\tree p \hx \ell$. 
\end{proof}

Again it is worth pointing out that Lemma~\ref{lem:q2} tells us that basically all that matters is the mod $p$ value of the prefactor.  We can now define a non-recursive way to compute the trees defined in Definition~\ref{def:trees}.  We now describe an algorithm that will compute this tree exactly. We will state but not prove the algorithm.

\newcommand{\U}[3]{Z_{#1}({#2},{#3})}
\begin{defn}\label{def:unfurl}
For any $p$ prime, and $x\in\{0,\dots, n-1\}$, we define the {\bf unfurled preimage of $x$ of depth $\zeta$}, $\U p x \zeta$, as a directed tree.
The vertices of $\U p x \zeta$ are all sequences of the form $(a_1,\dots, a_\eta)$ with $\eta \le \zeta$ such that $a_1 = x$ and $a_{i-1} = a_i^2 \bmod p$.
The edges of $\U p x \zeta$ are
\begin{equation*}
  (a_1, a_2,\dots, a_{\eta-1}, a_\eta)\to(a_1,a_2,\dots, a_{\eta-1})
\end{equation*}
\end{defn}

\begin{remark}
It is always true that $\U p x \zeta$ forms a tree, since each edge moves from a sequence to a strictly smaller sequence and thus cannot have loops.  (The construction above is sometimes referred to as the ``universal cover of the graph $\units{p}$ rooted at $x$''.
\end{remark}

\begin{prop}
  $\U p x \zeta$ is always a tree with degree $2$.  If $x$ is a transient point $\bmod\ p,$ then $\U p x \zeta$ is a regular $2$-tree.  If $p$ is prime, and $p-1=2^\theta \mu$ with $\mu$ odd, and $\zeta \le \mu$, then $\U p x \zeta$ is a regular $2$-tree.  Otherwise, it is not a regular $2$-tree.
\end{prop}

\begin{exa}
  Assume that $p\equiv 3 \bmod 4$, so that $p-1 = 2\mu$.  Let us first compute $\U p 1 \zeta$.  Note that $1$ has two preimages: $1$ and $-1$, so $\U p 1 1$ is a regular $2$-tree.  Now, if we want to go to depth $\zeta=2$, we see that $-1$ has no preimage, but $1$ again has two, so the $-1$ note terminates, but the $1$ node bifurcates.  Again, going to depth $\zeta=3$, we again terminate at $-1$ and bifurcate at $1$.  Depending on the depth to which we want to unfurl this tree, we can continue as long as we'd like.  Also note that we are labeling each node by the last term in the sequence defining the node, and not the whole sequence itself (it is useful to think of just this last value, but of course the sequence is important in the definition to give unique descriptors for each node).    See Figure~\ref{fig:tree}.
  
  Now, consider any periodic $x$.  We claim that $\U p x \zeta$ will be isomorphic to $\U p 1 \zeta$.  To see this, note that since $x$ is periodic, it has two preimages:  one is the periodic $y_1$ with $y_1\to x$, and the other is the preperiodic $z_1$ with $z_1\to x$, so we get a regular $2$-tree at depth 1.  To go to depth $\zeta=2$, we have that $y_1$ is periodic so it again bifurcates into the periodic $y_2$ and perperiodic $z_2$,  whereas $z_1$ is a leaf in $\gr p$, so it does not.  
\end{exa}

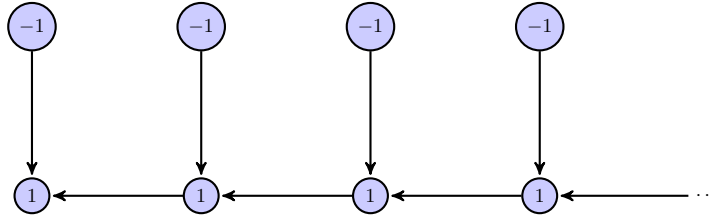
\begin{figure}
\begin{center}
\begin{tikzpicture}[->,>=stealth',shorten >=1pt,auto,node distance=3cm,
  thick,main node/.style={circle,fill=blue!20,draw,font=\sffamily\bfseries},transform shape, scale=0.75]

  \node[main node] (1) {$1$};
  \node[main node] (m1) [above  of=1] {$-1$};
  \node[main node] (11) [right of=1] {$1$};
  \node[main node] (m11) [above of=11] {$-1$};
  \node[main node] (111) [ right of=11] {$1$};
  \node[main node] (m111) [above of=111] {$-1$};
  \node[main node] (1111) [ right of=111] {$1$};
  \node[main node] (m1111) [above of=1111] {$-1$}; 
  \node(dots)[right of=1111]{$\cdots$};
  \path[every node/.style={font=\sffamily\small}]
    (m1) edge (1)
    (m11) edge (11)
    (m111) edge (111)
    (m1111) edge (1111)
    (11) edge (1)
    (111) edge (11)
    (1111) edge (111)
    (dots) edge (1111)
        ;
\end{tikzpicture}
\end{center}

\label{fig:tree}
\caption{Unfurling the case where $\theta=1$.}
\end{figure}

\begin{defn}[Layered trees and expansions]\label{def:expansion}
  We say that $T$ is a {\bf layered tree} with $L$ layers if $V(T)$ can be written as the disjoint union of $V_1, V_2,\dots, V_L$ and every edge in the graph goes from layer $k$ to layer $k-1$.  (This is basically the directed version of $L$-partite.)  Note that a vertex in a layer-$L$ tree is a leaf iff it is in layer $L$.

Given an $L$-layer tree $T$, and the integer vector ${\bf a} = (a_1,a_2,a_3,\dots,a_{L-1})$, we define $T^{({\bf a})},$ {\bf $T$ expanded by ${\bf a}$}, to be the following $L$-layer tree:

\be

\ii Let $v\in V(T)$ be a vertex in the $\ell$-th layer.  Then the $a_1\times a_2\times\cdots\times a_{\ell-1}$ vectors of the form
$$(v,(\alpha_1,\alpha_2,\dots, \alpha_\ell)), 1\le \alpha_i\le a_i$$
are all in the $\ell$-th later of $V(T^{({\bf a})})$.

\ii We define the edges of $T^{({\bf a})}$ by
\begin{equation*}
  (v,(\alpha_1,\alpha_2,\dots, \alpha_\ell)) \to (w,(\alpha_1,\dots, \alpha_{\ell-1})) \quad \iff \quad v\to w\mbox{ in } T.
\end{equation*}
\ee
\end{defn}

\begin{lem}\label{lem:layers}
  Let $n=p^k$, and let $x=\tx p^\ell$ for $\ell < k$, and $\hx = \tx \bmod p$.  Write $\ell = 2^\zeta \lambda$, where $\lambda$ is odd.   Define $\U p \hx \zeta$ as in Definition~\ref{def:unfurl} and expand this graph by $(p^{\ell/2}, p^{\ell/4}, p^{\ell/8},\dots, p^{2\lambda}, p^\lambda)$.  Then this graph is isomorphic to the basin of attraction of $x$ in $\nilp{p^k}$.
\end{lem}

\subsection{The graph $\gr{2^k}$}\label{sec:2^k}

We have dealt with all powers of odd primes above, but still need to deal with $n=2^k$.  This will be a bit different than what has gone before, both for the units and the nilpotents, but it will be similar enough that we can reuse some earlier ideas here.  First, a preliminary result that we can prove pretty simply.

\begin{prop}
  Let $n=2^k$.  Then $\gr{2^k}$ has exactly two components --- one corresponding to $\units{2^k}$ and the other to $\nilp{2^k}$, and they are both trees.  (Equivalently, the graph $\gr{2^k}$ has exactly two fixed points:  $0$ and $1$, and no periodic points.)
\end{prop}

\begin{proof}
  The proof of the claim for $\nilp{2^k}$ is similar to that for odd primes.  The nilpotent elements modulo $2^k$ are exactly the even numbers, and when these are raised to a high enough power, they will be 0 modulo $2^k$.  Now we consider the units, which are the odd numbers.  Let $x\equiv 1\bmod 2^\ell$ with $1\le \ell < k$.  Then $x^2 = 1 + \alpha 2^\ell$, and
  \begin{equation*}
    x^2 = (1 + 2\alpha 2^\ell + \alpha^2 2^{2\ell}) = 1+\alpha 2^{\ell+1} + \alpha^2 2^{2\ell} \equiv 1\pmod {2^\ell+1}.
  \end{equation*}
  Since the power has increased, from this we see that any odd number will be 1 modulo $2^k$ in at most $k-1$ steps, and we are done.
\end{proof}

From this we just need to characterize each of the two trees.  We consider $\units{2^k}$ first.  One complication here is that the units modulo $2^k$ do not form a cyclic group under multiplication, but we get something that is close enough to make it work:

\begin{lem}\label{lem:3}\cite{gauss1986english}
  The powers of 3 modulo $2^k$ form a cyclic group of size $2^{k-2}$ under multiplication.   If we write $S_k$ as the set 
  \begin{equation*}
    \{3^q\bmod 2^k: 0\le q < 2^{k-2}\},
  \end{equation*}
  then the units modulo $2^k$ (basically the odd numbers) can be written as $S_k \cup -S_k$.
\end{lem}

Using Lemma~\ref{lem:3}, we can characterize $\units{2^k}$:

\begin{prop}\label{prop:units2k}
  The graph of $\units{2^k}$ is a tree constructed in the following manner:  Start with the tree $T_2^k$, and for every node that is not a leaf, add two leaves directly mapping into it.
\end{prop}

\begin{proof}

Since $S_k$ forms a cyclic subgroup of the units of size $2^{k-2}$, the graph of these numbers can be realized as $\add{2^{k-2}}$, which by Corollary~\ref{cor:2^k} is $T_2^k$.  Note that all of the other units are negatives of powers of $3$, and under squaring the negative sign cancels.  So, for example, if $x\in\units{2^k}$ is such that $z_1^2\equiv  z_2^2\equiv x\bmod{2^k}$, then we also have $(-z_1)^2\equiv  (-z_2)^2\equiv x\bmod{2^k}$, so that $x$ now has in-degree four.  Moreover, if there is any leaf in the $\add{2^{k-2}}$, i.e. a number in $S_k$ that is not a square modulo $2^k$, then there can be no number in $-S_k$ that squares to it, since the negative of that number would be in $S_k$.  Therefore if a number is a leaf in the original tree, then it remains a leaf when we add in the $-S_k$ terms as well.
\end{proof}

See Figure~\ref{fig:units128} for an example with $n=2^7$.  We can also describe the graph $\units{2^k}$ intuitively as follows:  Start with the node at $1$ and add a loop;   add three nodes:  $2^{k}-1$ and $2^{k/2}-1$, which will be leaves, and the node $2^{k/2}+1$, which will itself be a tree; to $2^{k/2}+1$, we add four inputs:  the leaves $2^{k/4}-1$ and $3\cdot 2^{k/4}-1$, and two trees: $2^{k/4}+1$ and $3\cdot 2^{k/4}+1$, etc.

\begin{figure}[th]
\begin{center}
\end{center}
\includegraphics[width=0.95\textwidth]{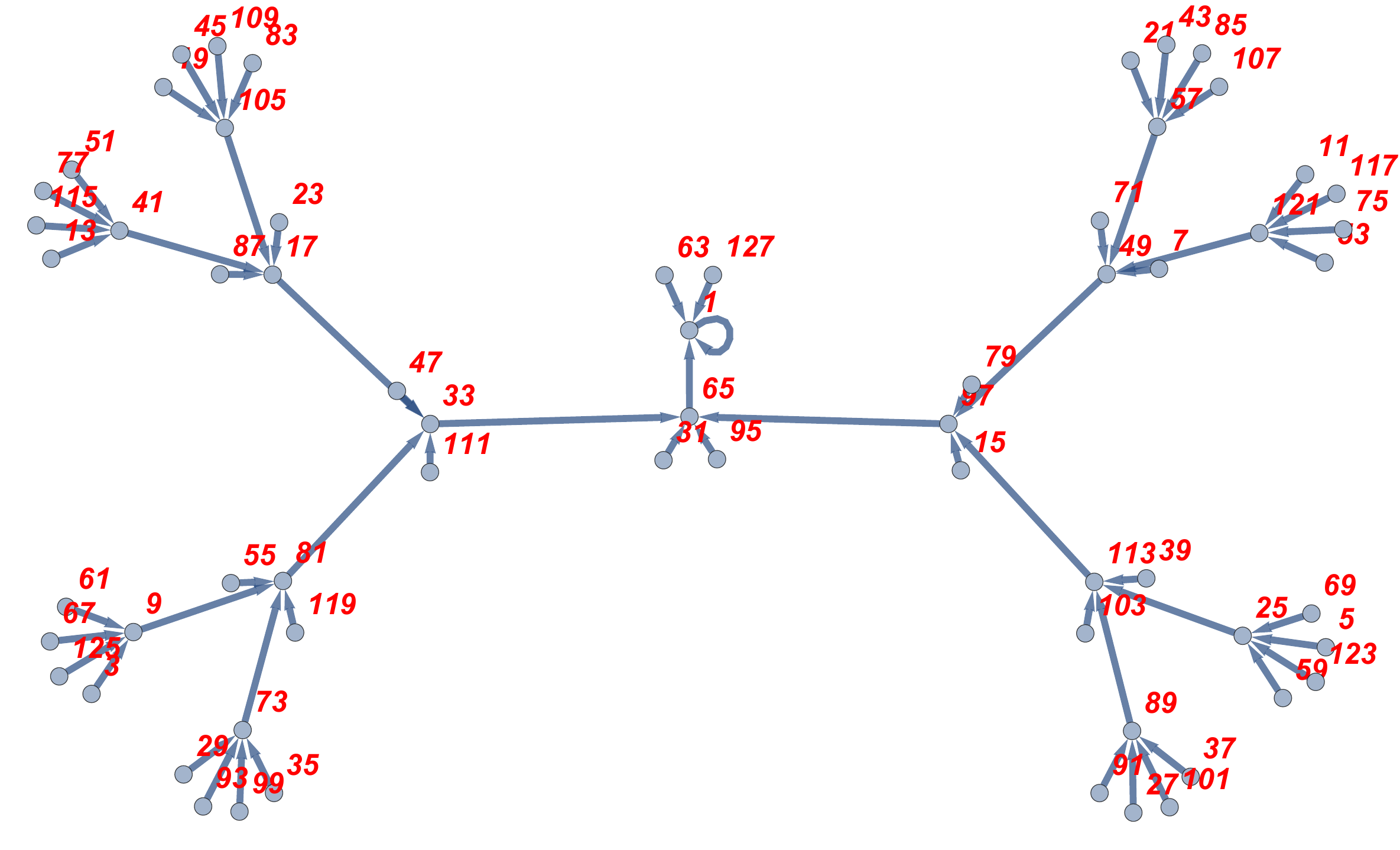}
\caption{The graph $\units{128}$.  Note that it is a graph of depth $5$ (because the cyclic subgraph generated by powers of $3$ is of order $32$, but instead of being a $T_2^5$, we add two leaves to each non-left node.  So for example without the additional leaves, $65 = 2^6+1$  would have two preimages: $33=2^5+1$ and $97 = 3\cdot 2^5+1$, which themselves are trees. But we also have the two leaves  $31 = 2^5-1$ and $95 = 3\cdot 2^5-1$ as well. (And similarly for all of the other nodes.)}
\label{fig:units128}
\end{figure}

Now onto the nilpotent elements.  This case is a bit more complex than the case where $p$ is odd.  The main problem here is that there is no statement analogous to Lemma~\ref{lem:q2} --- the main idea there is that the basin of attraction of any point of the form $x=\tx p^\ell$ depends only on the value of $\hx = \tx \bmod p$, whenever $p$ is odd.  This is clearly not true in the case where $p=2$, since this would mean that the basin of attraction for all numbers of the form $\mbox{(odd)}*2^\ell$ would be the same, but this cannot be true since some odd numbers have square roots modulo $2^k$ and some do not (q.v.~the discussion of units above).  In particular, when computing the basin of attraction of a point $x = \tx p^\ell$, we didn't need to pay attention to the ``ambient space'' $p^k$, but unfortunately this breaks when $p=2$.    So in theory one can compute this tree for any given $n$, but we don't expect such a nice result as seen in Theorem~\ref{thm:nilpotent} or Lemma~\ref{lem:layers}.  We do have a partial result that we state without proof.

\begin{prop}
  Consider $n=2^k$, and let $x=\tx 2^\ell$, where $\tx\neq 1$.  Let us define $\theta_1,\theta_2$ by the equations
  \begin{equation*}
    \tx - 1 = 2^{\theta_1} \cdot q_1,\quad \ell = 2^{\theta_2}\cdot q_2,
  \end{equation*}
  where $q_1,q_2$ are odd.   Define $\theta = \min(\theta_1,\theta_2)$.  Define $T_{2}^\theta$ as in Definition~\ref{def:tree}, add on the extra nodes as we did in Proposition~\ref{prop:units2k}, and expand it as in Definition~\ref{def:expansion}, where we expand the first layer by $2^{\ell/2}$, the second layer by $2^{\ell/4}$, etc.  Then the basin of attraction of $x$ is isomorphic to this expanded graph.
\end{prop}

\section{Specific Computations}\label{sec:computations}

\subsection{Computing the Kronecker products in practice}

As we have seen above, the type of graph structures we see for primes, or powers of primes, come in certain forms --- and the graphs for composite $n$ come from Kronecker products of these.  We have two main results in this section:  the first is that one can compute the Kronecker product ``component by component'', and the second one is a list of the typical graph products will we see in practice.

\begin{defn}
  If $G$ are $H$ are two graphs with disjoint vertex sets, then we define the {\bf disjoint union of $G$ and $H$}, denote $G\oplus H$, as the graph with $V(G\oplus H) = V(G) \cup V(H)$ and $E(G\oplus H) = E(G)\cup E(H)$.
\end{defn}

Now, let us assume that there is a path from $(x_1,y_1)$ to $(x_2,y_2)$ in $G\kgp H$.  If we ignore the $y$'s, this gives a path from $x_1$ to $x_2$ in $G$, and if we ignore the $x$'s, this gives a path from $y_1$ to $y_2$ in $H$.  Therefore if $(x_1,y_1)$ and $(x_2,y_2)$ are in the same component of $G\kgp H$, then $x_1$ and $x_2$ must be in the same component in $G$, and $y_1$ and $y_2$ must be in the same component in $H$.  This gives the following:

\begin{prop}\label{prop:components}
  Let $G$ and $H$ be graphs, and let 
 \begin{equation*}
   G = \bigoplus_{i\in I} G_i,\quad\quad 
   H = \bigoplus_{j\in J} H_j
 \end{equation*}
  be the decomposition of each graph into its connected components.   Then we have
  \begin{equation*}
    G\kgp H = \bigoplus_{i\in I, j\in J} (G_i\kgp H_j),
  \end{equation*}
  where the individual terms in the union are themselves disjoint. More specifically, if there is a path from $(x_1,y_1)$ to $(x_2,y_2)$ in $G\kgp H$, then there must be a path from $x_1$ to $x_2$ in $G$ and one from $y_1$ to $y_2$ in $H$.
\end{prop}

  Basically this means that we can compute the Kronecker product ``component by component'' --- more specifically, when computing the Kronecker product we can just look at each component individually and not worry about the impact from other components.  This is good, because we can compartmentalize.  However, one caveat, as we will see below:  it is possible for the product of two connected graphs to not be connected, so that some of the $G_i\kgp H_j$ terms in the expansion above might not themselves be a single components, but can be multiple components.  However, we do see that the number of connected components of $G\kgp H$ is bounded below by the product of the numbers of connected components of $G$ and $H$, respectively.  Ok, so what do these component-by-component products actually look like?

As we learned from Theorems~\ref{thm:units} and~\ref{thm:nilpotent}, every component of $\gr{p^k}$ is either a flower cycle, or a looped tree (recall Definition~\ref{def:flower-cycle}).  This, with Proposition~\ref{prop:components}, tells us that the natural question is what happens when we take the Kronecker product of two graphs which are each of one of these two types.  We categorize all of the cases we might expect to see later in the following Proposition.

\begin{prop}\label{prop:kron2}
We have the following results:

\be

\ii (Two looped trees) If $T_1$ and $T_2$ are looped trees, then so is $T_1\kgp T_2$ --- specifically, this implies that $T_1\kgp T_2$ is connected, has a single root vertex, that root vertex loops to itself, and all vertices flow toward that root.  

\ii (Degrees)  Assume that $T_1$ and $T_2$ are looped trees with the property that every vertex in $T_i$ either has in-degree $w_i$ or $0$.  Then every vertex in $T_1\kgp T_2$ has in-degree $w_1\times w_2$ or $0$.

\ii (Specific grounded trees) For any widths $w_1,w_2$ and depth $\theta$, we have
\begin{equation*}
  T_{w_1}^\theta\kgp T_{w_2}^\theta \cong T_{w_1\cdot w_2}^\theta.
\end{equation*}
(If we take two regular grounded trees {\bf of the same depth}, then we get a regular grounded tree of that depth, just wider.)

\ii (Another description of flower cycles)  
\begin{equation*}
  C_\alpha\kgp T \cong C_\alpha(T).
\end{equation*}

\ii (One looped tree and one flower cycle)  \begin{equation*}
  C_\alpha(T_1)\kgp T_2 \cong C_\alpha(T_1\kgp T_2).
\end{equation*}

\ii (Two flower cycles)  Let $G = C_\alpha(T_1)$ and $H = C_\beta(T_2)$.  Let $\gamma = \gcd(\alpha,\beta)$ and $\lambda = \lcm(\alpha, \beta)$.  Then $G\kgp H$ is $\gamma$ disjoint copies of the graph $G_\lambda(T_1\kgp T_2)$.
\ee
\end{prop}

Note the last case where two connected graphs have a disconnected product!
\begin{proof}

\be

\ii  Consider the vector $(0,0)$ in $V(T_1)\times V(T_2)$, where each $0$ corresponds to the root in that tree.  Then note that $(0,0)\to(0,0)$ by definition, and we have this loop.  More generally, if we take any $(x,y)\in V(T_1)\times V(T_2)$, then there is a finite path from $x$ to $0$ in $T_1$ and a finite path from $y$ to $0$ in $T_2$, so there is a finite path (whose length is the maximum of the two aforementioned paths) from $(x,y)$ to $(0,0)$.

\ii In fact we prove something a bit more general:  if $x\in V(G)$ has an in-degree of $k$ and $y\in V(H)$ has an in-degree of $l$, then $(x,y)\in V(G\kgp H)$ has an in-degree of $k\times l$.  To see this, note that if $z_i \to x$ for $i=1,\dots, k$ and $w_j\to y$ for $j=1,\dots, l$, then $(z_i, w_j)\to (x,y)$ for all $i,j$.  The claim follows directly.

\ii Let us consider the case where we have two trees $G_1 = T_{w_1}^\theta$ and $G_2 = T_{w_2}^\theta$.  See Lemma~\ref{lem:param} below for a parameterization of each of these trees.  Then, the root of $G_1\kgp G_2$ is $(0,0)$ where $0$ is the root of each $G_1$.  Note that it has a loop to itself.  Let us define the first layer of $G_1\kgp G_2$ to be any pair $(k_1,k_2)$ with $0\le k_1\le w_1-1$ $0\le k_2\le w_2-1$ but not both zero.  Note that all of these map to $(0,0)$ in one step, and there are exactly $w_1w_2-1$ of these.  Now we define the $\ell$th layer of $G_1\kgp G_2$:   given pair $(v_1,v_2)$ where $v_1$ is in level $\ell_1$ of $G_1$ and $v_2$ is in level $\ell_2$ of $G_2$, define $\ell = \max(\ell_1, \ell_2)$.  Note that any such vertex will map down to layer $\ell-1$.  Moreover, if $\ell<\theta$, then each of these vertices has exactly $w_1w_2$ preimages, because we can append $w_1$ numbers to the end of $v_1$ and $w_2$ numbers to the end of $v_2$.  If $\ell=\theta$, then these elements are all leaves, and we are done.

\ii  This is actually a bit trickier than it looks, since these two descriptions are isomorphic but not equal.  Recall the definition of $C_\alpha(T)$:  we define the vertices of this graph to be $(v,\beta)$ where $v\in V(T)$ and $1\le \beta\le \alpha$.  The edges in $ C_\alpha(T)$ are defined as follows:
\begin{equation*}
  (v,\beta)\to (w,\gamma) \iff (\beta=\gamma \land v\to w) \lor (v=w=r \land \gamma = \beta+1).
\end{equation*}
However, in $C_\alpha\kgp T$ (actually $T\kgp C_\alpha$ but who's counting) we have 
\begin{equation*}
  (v',\beta') \to (w',\gamma') \iff v' \to w' \land \gamma' = \beta' +1.
\end{equation*}
Now let us define a map $\psi\colon C_\alpha(T) \to C_\alpha \kpg T$ as $\psi(v, \beta) = (v, \beta-\mathcal L(v))$, where $\mathcal L (v)$ is the level of the vertex $v$ (in this case, the number of edges it takes to get from $v$ to $r$).  Now, let us assume that $\phi(v,\beta)\to \psi(w,\gamma)$ in $C_\alpha \kpg T$.  This is true iff $(v,\beta-\mathcal L(v))\to (w,\gamma-\mathcal L(w))$ in $C_\alpha \kpg T$, which if true iff $v\to w$ in $T$ and $\gamma-\mathcal L(w) = \beta-\mathcal L(v) + 1$.  Now we claim that this is equivalent to $(v,\beta)\to(w,\gamma)$ in $C_\alpha(T)$.  Breaking up by cases:  if $v\neq  w$, then $\mathcal L(w) = \mathcal L(v) = 1$, which would imply $\gamma = \beta$, and if $v=r$, then $w=r$, and $\mathcal L(v) = \mathcal L(w) = 0$, which would imply $\gamma =\beta+1$.

\ii This follows from the previous result twice and associativity, since
\begin{equation*}
  C_\alpha(T_1) \kgp T_2 = (C_\alpha \kgp T_1)\kgp T_2 = C_\alpha \kgp (T_1\kgp T_2) = C_\alpha(T_1\kgp T_2).
\end{equation*}

\ii  Let us first note that if we consider two bare cycles, we get a similar formula:  Let $G_1 = C_\alpha$ and $G_2 = C_\beta$.  Note that we can think of these cycles as the map $x\mapsto x+1$ on $\zmod \alpha$ and $\zmod\beta$ respectively, and here $G_1\kgp G_2$ will be the graph of the map $(x_1,x_2) \mapsto (x_1+1,x_2+1)$ on $\zmod\alpha\times\zmod \beta$.  It is easy enough to see that the element $(1,1)$ has order $\lambda$, and so every point is on a periodic cycle of length $\lambda$.  Since there are $\alpha\cdot\beta$ total vertices, and $\alpha\cdot\beta/\lambda = \gamma$, there are exactly $\gamma$ distinct periodic cycles.

To get the full case: we have
\begin{equation*}
  C_\alpha(T_1)\kgp  C_\beta(T_2) = (C_\alpha\kgp T_1)\kgp (C_\beta \kgp T_2) = (C_\alpha\kgp C_\beta)\kgp(T_1\kgp T_2),
\end{equation*}
and we have computed above that $C_\alpha\kgp C_\beta$ is $\gamma$ copies of $C_\lambda$.

\ee

\end{proof}

\begin{lem}\label{lem:param}
We can parameterize $T_w^\theta$ as follows:  let $V_0 = \{0\}$, $V_1$ be the set $\{1,\dots, w-1\}$ and let 
\begin{equation*}
  V_\ell = \{(v_1,v_2,\dots, v_\ell): 1 \le v_1 \le w-1, \forall k > 1, 1 \le v_k \le w\}.
\end{equation*}
(These correspond to the ``layers'' in the graph, so any vertex in $V_\ell$ is ``in layer $\ell$''.)  Let $V = \cup_{i=0}^\theta V_\ell$.  Now define edges as follows:  for all $k\in V_1$, add an edge $k\to 0$, and for all $v\in V_\ell$ with $2\le \ell \le \theta$, add the edge
\begin{equation*}
  (v_1,v_2,\dots, v_{\ell-1}, v_\ell) \to (v_1,v_2,\dots, v_{\ell-1}),
\end{equation*}
i.e. throw out the last component.  Then this graph is isomorphic to $T_w^\theta$.
\end{lem}

\subsection{The same unit graph can show up in many places}

A natural question is how similar the graphs can be for different $n$. However, note that if we have two primes $p,q$ with $p\neq q$, then $\phi(p)\neq\phi(q)$, so they do not have isomorphic sets of units.  But it is possible that we can have two primes $p\neq q$ with $\phi(p^k) = \phi(q)$, and thus $\units{p^k} \cong \units{q}$.  For example, we have that
\begin{equation*}
  \phi(27) = \phi(19) = 18,
\end{equation*}
so $\units{27}\cong\units{19}\cong\add{18}$. We now work both cases out.  Since $18 = 2\cdot 9$, we have $\theta=1$ (so the flowers on the cycle are just one node sticking out) and the divisors are $1,3,9$.  Since $\phi(3) = \ordt 3 = 2$, there is one orbit of period $2$, and since $\phi(9) = \ordt 9 = 6$, there is one orbit of period $6$.  This determines the graph of units completely.  For $n=19$, of course, we have that $\nilp{19}$ is just $0$ with a loop to itself, but $\nilp{27}$ is more complicated.  The set of points that map to zero modulo $27 = 3^3$ are the two points $9=3^2$ and $18=2\cdot 3^2$.  If we recall the graph $\gr 3$, we see that $1$ has a preimage of $2$, and $2$ has no preimage.  Therefore $18$ is a leaf in $\nilp{27}$, whereas $9$ is a tree of the form $\tree 3{1}2$, which will be a regular tree of width $6$ and depth $1$ (there are three trees of type $\tree 3 1 1$ and three of type $\tree 3 2 1$, which are themselves single nodes by definition, so are leaves in $\nilp{27}$).  See Figure~\ref{fig:18}.

\begin{figure}[th]
\begin{center}
\end{center}
\includegraphics[width=0.95\textwidth]{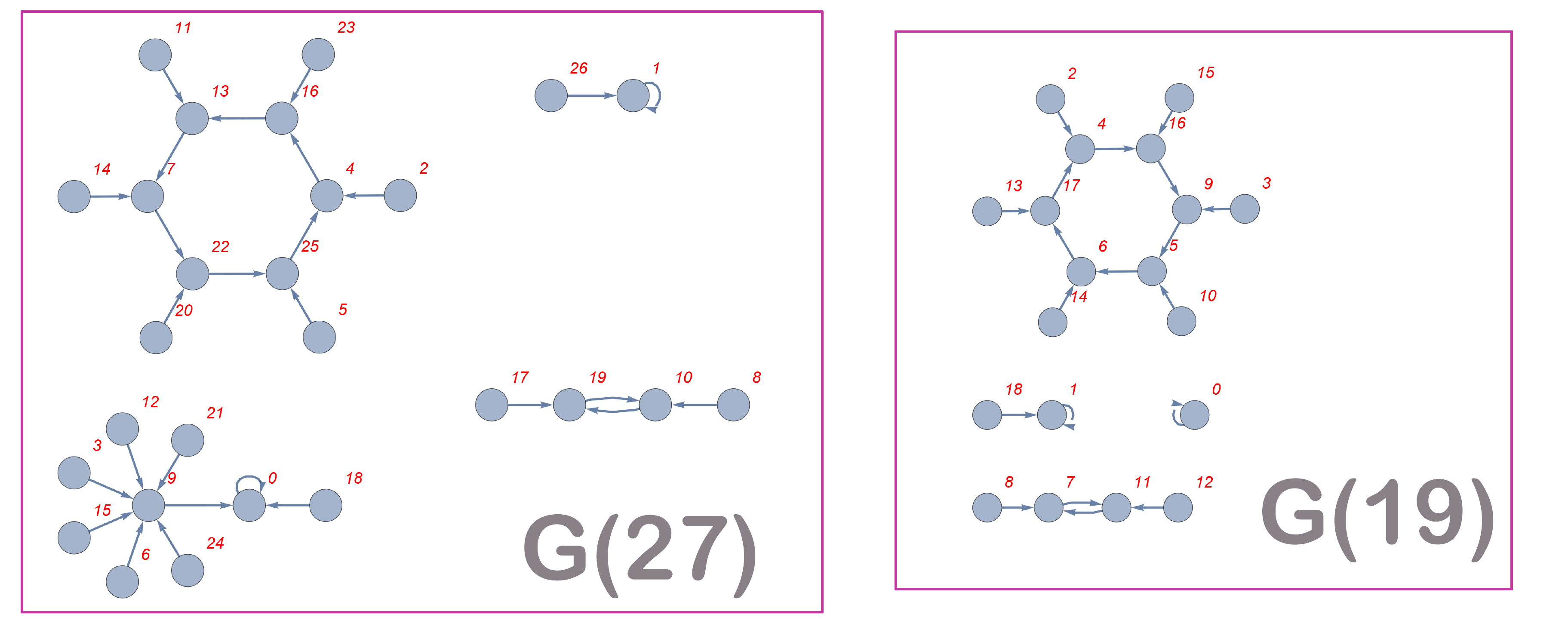}
\caption{The graphs $\gr{27}$ and $\gr{19}$.  Note that they are isomorphic on the units (the graphs are the same, but of course the values at the vertices are different), with only a different nilpotent tree.}
\label{fig:18}
\end{figure}

\subsection{Example where we really go to town on the Kronecker products}

As we proved in Proposition~\ref{prop:kron2}, all of the components that show up in $\gr n$ are themselves Kronecker products of graphs that show up in the factors of other graphs.  So, let's say, for example, we want to find an $n$ that contains a copy of the graph
\begin{equation*}
  T_2^1 \kgp T_2^2 \kgp T_2^3 \kgp T_2^4.
\end{equation*}
One way to obtain this is to find primes that contain such graphs, and then multiply these together.  We know that a $\gr p$ for $p$ prime will contain a $T_2^\theta$ iff $p-1$ is divisible by exactly $\theta$ powers of $2$.   For $\theta=1$, we can pick $p_1=3$, and for $\theta=2$ we can pick $p_2=5$.  For $\theta=3$ we cannot pick $2^3+1$, since it's not prime, and we don't want to pick any even multiple of $8$, as we'll get too many powers of $2$.  So we can pick $p_3=41$, and finally we can pick $p_4= 17$.  Of course, $p_1$, $p_2$, and $p_4$ are relatively small since they are Fermat primes, but we need to stretch a bit for $p_3$.

Now we know that the basin of attraction of $1$ in $\gr{p_i}$ is exactly $T_2^i$, for $i=1,2,3,4$.  We have $p_1p_2p_3p_4 = 10455$, and therefore the basin of attraction of $1$ in $\gr{10455}$ is exactly $T_2^1 \kgp T_2^2 \kgp T_2^3 \kgp T_2^4$, shown in Figure~\ref{fig:1024attractor}.  (In fact, to generate this picture we actually just directly computed the subgraph of $\gr{10455}$ containing $1$ directly.)  It's a pretty wild graph:  it has a radius of $4$ due to the $T_2^4$ term, but the smaller terms give it some interesting heterogeneity (e.g. each intermediate node has leaves hanging off of it, etc.).

\begin{figure}[th]
\begin{center}
\includegraphics[width=0.65\textwidth]{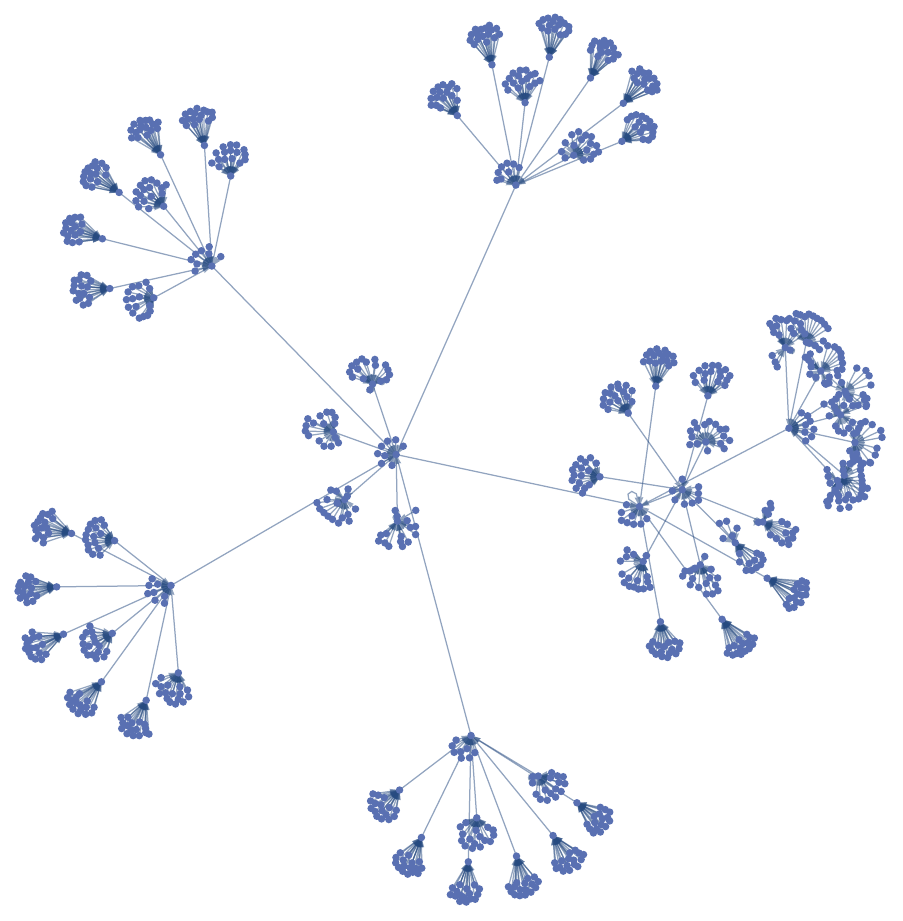}
\caption{The graphs $T_2^1 \kgp T_2^2 \kgp T_2^3 \kgp T_2^4$, which appears as the basin of attraction of $1$ in the graph $\gr{10455}$.  At this point, it is perhaps clear why we gave ``flower cycles'' their names:  the }
\label{fig:1024attractor}
\end{center}
\end{figure}

In fact, we can completely understand the graph of $\gr{10455}$ using the tools above.  Note that since $p_1,p_2,p_4$ are all Fermat primes, their graphs consist of a tree containing $p_i-1$ nodes, plus a $0$ with a loop, so is the union $T_2^i \cup T_1^1$.  For $p_3=41$, we note that $p-1=40 = 8*5$, and $\ordt 5 =\phi(5)= 4$, so $\gr{41}$ has a single loop of period $4$ in the form of a $C_4(T_2^3)$, plus a $T_2^3$ going to $1$, plus the single loop at $0$.  There are 24 possible products in the expansion
\begin{equation*}
  (T_2^1 \cup T_1^1)\kgp (T_2^2 \cup T_1^1)\kgp(C_4(T_2^3) \cup T_2^3 \cup T_1^1) \kgp (T_2^4 \cup T_1^1)
\end{equation*}
and notice that there is only one loop that can be chosen in any of these products, so each product remains connected and there are precisely 24 components.  Moreover, exactly eight of these will have a loop of period $4$, and sixteen will not, for example for loops we can get a $C_4(T_2^3)$ by choosing $T_1^1\kgp T_1^1 \kgp C_4(T_2^3) \kgp T_1^1$, but we can also get $C_4(T_2^1\kgp T_2^3)$ by choosing $T_2^1\kgp T_1^1 \kgp C_4(T_2^3) \kgp T_1^1$, etc.

In general, we can extrapolate directly from this example that for any flower cycle of the form $C_\alpha(T)$, where $T$ is a tree that can be formed by taking Kronecker products of any collection of $T_2^\theta$'s, then we can find a (square-free) $n$ that contains $C_\alpha(T)$.

\subsection{Triggering divisors}\label{sec:triggering}

\begin{defn}
  Given a period $k$, we say that $d$ is a {\bf triggering divisor} for $k$ if 
  \bi
  
  \ii $\ordt d = k$;
  
  \ii $\ordt e\neq k$ for any $e$ that is a proper divisor of $d$.
  \ei
\end{defn}

\begin{remark}
It follows directly from the definition that   $\ordt d = k$ iff $d$ is a multiple of a triggering divisor of $k$.

The name ``triggering divisor'' comes from the fact that the presence of such a divisor of $m$ will trigger a periodic orbit of period $k$.  (It might be that other divisors of $m$ also give periodic orbits of period $k$, but this minimal one is the one that triggers the condition...)
\end{remark}

\begin{cor}
  The graph $\gr{p^k}$ has a flower orbit of type $C_k(T_2^\theta)$ iff $m = \phi(p^k)$ is a multiple of $2^\theta$ times a triggering divisor of $k$.
\end{cor}

\begin{proof}
  This follows directly from Theorem~\ref{thm:units} and the definition of triggering divisor.
\end{proof}

We present the triggering divisors for all periods up to $50$ in Table~\ref{tab:trig}, and we have also put the first $n$ for which a period appears.  Note that the triggering divisors are not ``unique'' in some cases for certain periods --- we can obtain multiple numbers, neither of which divides the other.  Also note, interestingly, that these numbers might not even be relatively prime.  In the table, when there are multiple triggering divisors we have bolded the one that gives rise to the first prime to have that period; for example, for period 18, the smallest prime that is $1\pmod{19}$ is actually 191, whereas the smallest prime that is $1\pmod{27}$ is 109. 

Also, one might ask whether a composite number gives rise to a particular periodic orbit before any prime does.  In theory it is possible that one could observe a (composite) period first in a composite number.  For example, to obtain a period-10 orbit, we can find a prime that gets it from a triggering divisor (in this case, 47) or we could find a number with a period-2 and a period-5 and multiply them.  So, for example, $n=7*311 = 2177$ also has a period-10 orbit, but $2177 > 47$.  In practice we never found an example where a period first shows up in a composite number but there are a lot of integers that we didn't check.  The author is agnostic on the claim as to whether a particular period length can ever show up first at a composite $n$.

There is a lot of heterogeneity in how large the triggering divisors are.  One thing that follows directly from the definitions is that
\begin{prop}\label{prop:Mersenne}
  If $p$ is a Mersenne prime ($p$ is prime and $2^p-1$ is also prime) then period-$p$ has the single triggering divisor $2^p-1$.
\end{prop}
One can see this markedly at $2, 3, 5, 7, 13, 17, 19, 31$ in Table~\ref{tab:trig}, whereas for primes that are not Mersenne primes (e.g. $p=11, 37, 39$ there is a ``pretty small'' triggering divisor.)

The method to obtain the values in Table~\ref{tab:trig} was as follows.  For any period $r$, we list all divisors of the number $2^r-1$.  For any $d$ that divides $2^r-1$, we compute the multiplicative order of $2$ modulo $d$ and retain those where $\ordt d = r$.  (Of course, since $d$ divides $2^r-1$, then $\ordt d$ must divide $r$, but it could be strictly smaller.)  From this retained list of divisors, we remove those divisors that are multiples of others in the list.
For example, let us consider period $6$.  We have $2^6-1 = 63$, the divisors of which are $1,3,7,9,21,63$.  We have $\ordt 1=1,\ordt3=2$ and $\ordt7=3$ and remove those.  Of the remaining three, we can compute that they satisfy $\ordt d = 6$.  But then notice that $63$ is a multiple of $9$ (and also $21$ for that matter) so we throw out $63$, leaving $9$ and $21$.  And also note that these divisors are independent in the sense that they give different sets of primes.  For example, if we consider $p=19$, then $19-1 = 18 = 2*9$, which is not divisible by $21$, and conversely $p=43$ has $p-1 = 2*21$ which is not divisible by $9$.  Note that $\theta=1$ in each of these cases, so we see that $\gr{19}$ has a single period-6 orbit of type $C_6(T_2^1)$ (since $\phi(9) = 6$) but $\gr{43}$ has two disjoint such orbits (since $\phi(21) = 12$).  See Figure~\ref{fig:1943}.

\begin{table}
\begin{center}
\begin{tabular}{|c|c|c|c|c|c|}\hline 
Period & Trig div &First&Period & Trig div & First \\\hline
1 & 1 &  1& 26 & {\bf 2731},24573 &60083\\\hline 
 2 & 3 & 7& 27 & 262657 &2626571\\\hline 
 3 & 7 & 29 &28 & {\bf 29},113,215,635 & 59 \\\hline 
 4 & 5 & 11& 29 & {\bf 233},1103,2089 & 467 \\\hline 
 5 & 31 & 311& 30 & 77,{\bf 99},279,331,453,651,1661 &199\\\hline 
 6 & {\bf 9}, 21 & 19& 31 & 2147483647 & ?? \\\hline 
 7 & 127 & 509&32 & 65537 & 917519\\\hline 
 8 & 17 & 103 & 33 & {\bf 161},623,599479 & 967 \\\hline 
 9 & 73 & 293&34 & {\bf 43691},393213 &87383\\\hline 
 10 & {\bf 11},93 & 23& 35 & {\bf 71},3937,122921 & 569\\\hline 
 11 & {\bf 23},89 & 47&36 & {\bf 37},95,109,135,247,351,365,949 & 149 \\\hline 
 12 & {\bf 13},35,45 & 53& 37 & {\bf 223},616318177 & 2677 \\\hline 
 13 & 8191&376787 & 38 & {\bf 174763},1572861 & 699053\\\hline 
 14 & {\bf 43},381 & 173&39 & {\bf 79},57337,121369 & 317 \\\hline 
 15 & {\bf 151},217 & 907&40 & {\bf 187},425,527,697,61681  & 1123\\\hline 
 16 & 257 & 1543&41 & {\bf 13367},164511353 & 106937\\\hline 
 17 & 131071 &1572853& 42 & {\bf 147},301,387,1011,1143,2667,5419,14491 & 883\\\hline 
 18 & 19,{\bf 27},219 & 109& 43 & {\bf 431},9719,2099863 & 863 \\\hline 
 19 & 524287 & 8388593&44 & {\bf 115},397,445,2113,3415 & 461\\\hline 
 20 & 25,{\bf 41},55,155 & 83&45 & {\bf 631},2263,11023,23311 & 6311 \\\hline 
 21 & {\bf 49},337,889 & 197& 46 & {\bf 141},535443,2796203 & 283\\\hline 
 22 & {\bf 69},267,683 & 139& 47 & {\bf 2351},4513,13264529 & 4703 \\\hline 
 23 & {\bf 47},178481 & 283 &48 & {\bf 97},673,1799,2313,3341,61937  & 389\\\hline 
 24 & {\bf 119},153,221,241 & 239&49 & 4432676798593 & ??\\\hline 
 25 & {\bf 601},1801 & 3607&50 & {\bf 251},1803,4051,5403,6611,19811 &503\\\hline 
\end{tabular}
\caption{List of triggering divisors for period up to 50}
\end{center}
\label{tab:trig}
\end{table}

\begin{figure}[th]
\begin{center}
\includegraphics[width=0.95\textwidth]{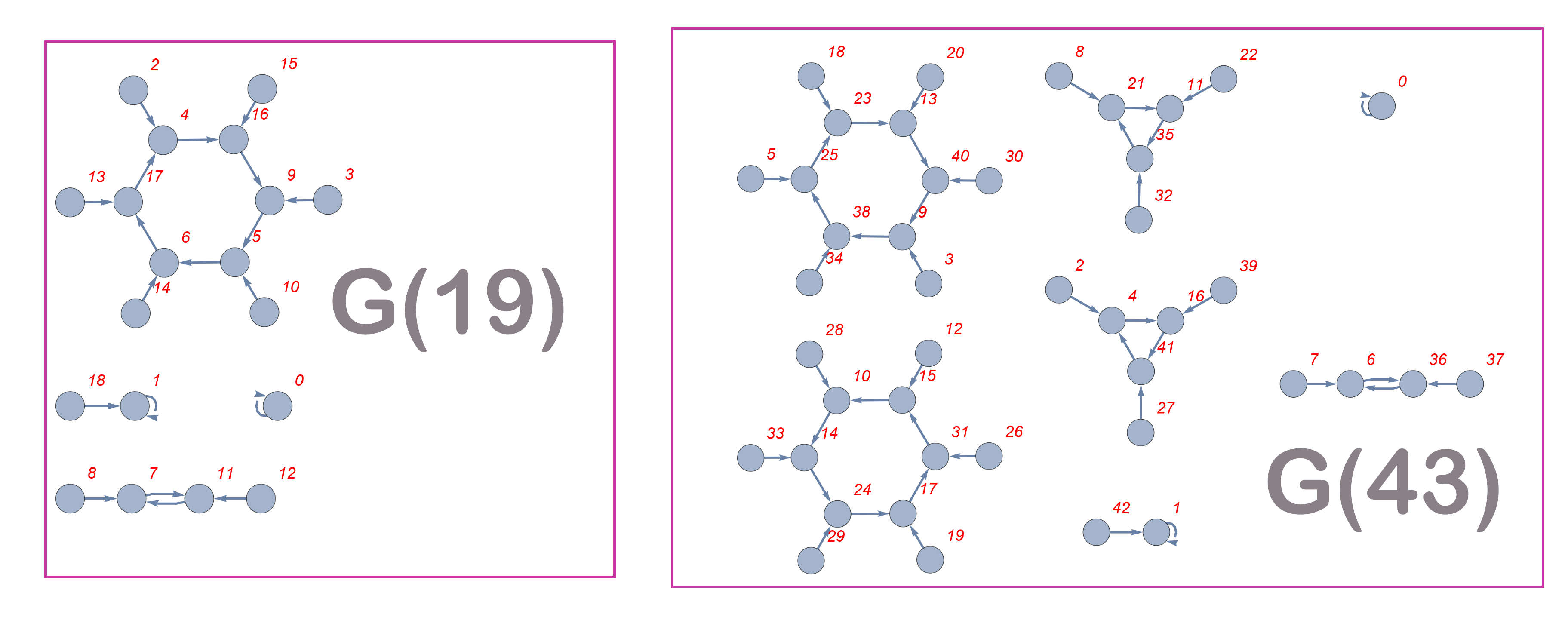}
\caption{The graphs $\gr{19}$ and $\gr{43}$.  Note that they both have period-6 orbits, and all periodic orbits have a ``spoke'' sticking out from the $T_2^1$.}
\label{fig:1943}
\end{center}
\end{figure}

\subsection{When cycles beget many more cycles}

Consider the two numbers considered in the previous section: the primes $19$ and $43$.  These numbers were chosen as primes that give period-6 orbits, but ``for different reasons'', because they derive from independent triggering divisors.  Since they both have periodic orbits of the same period, if we take the Kronecker product, we will obtain many orbits of that period, from Proposition~\ref{prop:kron2}, since $\alpha=\beta = \lambda = \gamma$ --- so that the product of two period-6 cycles is actually 6 distinct period-6 cycles.

Again, we first consider $\gr{19}$. Since $19-1 = 18 = 2*9$, we need to consider the divisors of $9$.  We already saw that we have an orbit of period $6$, and from the divisor $3$ we get and orbit of period $2$. Since $\theta=1$ we have $T_2^1$ attached to the graphs (just giving a ``spoke'').  Therefore $\gr{19}$ has four components, namely $$\gr{19}\cong C_6(T_2^1) \oplus C_2(T_2^1) \oplus T_2^1 \oplus T_1^1.$$

For $\gr{43}$, we have $43-1 = 42 = 2*21$, and so we need to consider the divisors $1,3,7,21$.  As we computed above, the $d=21$ gives two period-6 orbits.  Since $\phi(7) = 6$ and $\ordt 7 = 3$, the $d=7$ term gives two period-3 orbits, and $d=3$ gives one period two orbit.  Therefore $\gr{43}$ has 7 components:
$$\gr{43} \cong C_6(T_2^1) \oplus C_6(T_2^1) \oplus C_3(T_2^1)\oplus C_3(T_2^1) \oplus C_2(T_2^1) \oplus T_2^1 \oplus T_1^1.$$

Now let us consider $n=817 = 19*43$.  We know that the graph $\gr{817}$ will contain at least $28$ components, since one of the graphs has 4 and the other 7.  But in fact, it will have many more components:  each time we take the product of two period-6 orbits, we actually get {\em six} period-6 orbits, etc.

For example, let us consider the $C_6(T_2^1)$ component of $\gr{19}$, and multiply it against the seven components of $\gr{43}$.  We get
\begin{align*}
  C_6(T_2^1) \kgp C_6(T_2^1) &= C_6(T_4^1) \quad\mbox{times 6}; \\
  C_6(T_2^1) \kgp C_3(T_2^1) &= C_6(T_4^1) \quad\mbox{times 3};\\
  C_6(T_2^1) \kgp C_2(T_2^1) &= C_6(T_4^1) \quad\mbox{times 2};\\
  C_6(T_2^1) \kgp T_2^1 &= C_6(T_4^1) \quad\mbox{times 1};\\
 C_6(T_2^1) \kgp T_1^1 &= C_6(T_2^1).   
\end{align*}
And also note that the first two components listed above each appear twice in $\gr{43}$, so when we multiply $C_6(T_2^1)$ by the components of $\gr{43}$ we obtain $2*6+2*3+2+1 = 21$ copies of $C_6(T_4^1)$ and one copy of $C_6(T_2^1)$.

When we multiply the $C_2(T_2^1)$ component of $\gr{19}$ against the other seven, we obtain
\begin{align*}
  C_2(T_2^1) \kgp C_6(T_2^1) &= C_6(T_4^1) \quad\mbox{times 2}; \\
  C_2(T_2^1) \kgp C_3(T_2^1) &= C_6(T_4^1) \quad\mbox{times 1};\\
  C_2(T_2^1) \kgp C_2(T_2^1) &= C_2(T_4^1) \quad\mbox{times 2};\\
  C_2(T_2^1) \kgp T_2^1 &= C_2(T_4^1) \quad\mbox{times 1};\\
 C_2(T_2^1) \kgp T_1^1 &= C_2(T_2^1).   
\end{align*}
This gives a total of $2*2+2*1 = 6$ copies of $C_6(T_4^1)$, one copy of $C_2(T_4^1)$, and one copy of $C_2(T_2^1)$.

The other two components are simpler, since they are trees.  Multiplying $T_2^1$ against the components of $\gr{43}$ gives $2$ copies of $C_6(T_4^1)$, $2$ copies of $C_3(T_4^1)$, $1$ copy of $C_2(T_4^1)$, $1$ copy of $T_4^1$, and finally one copy of $T_2^1$.  Multiplying $T_1^1$ against these components just copies them.

Putting this all together, we have

\begin{center}
\begin{tabular}{|c|c|}\hline
Graph type & copies\\\hline
$C_6(T_4^1)$ & 21 + 6 +2 = 29\\\hline
$C_3(T_4^1)$ & 2\\\hline
$C_2(T_4^1)$ & 1+1=2\\\hline
$C_6(T_2^1)$& 1 + 1 +1  = 3\\\hline
$C_3(T_2^1)$& 1 + 1 + 1 = 3\\\hline
$C_2(T_2^1)$& 1 + 1 + 1 = 3\\\hline
$T_4^1$ & 1\\\hline
$T_2^1$ & 1+1=2\\\hline
$T_1^1$ & 1\\\hline
\end{tabular}
\end{center}

We can compare this against the graph given in Figure~\ref{fig:817}, which was obtained by direct computation.  Note that $C_\alpha(T_4^1)$ will be a cycle of period $\alpha$ with three leaves sticking out of each node, and  $C_\alpha(T_2^1)$ will be a cycle of period $\alpha$ with a single stem out of each node.

\begin{figure}[th]
\begin{center}
\includegraphics[width=0.85\textwidth]{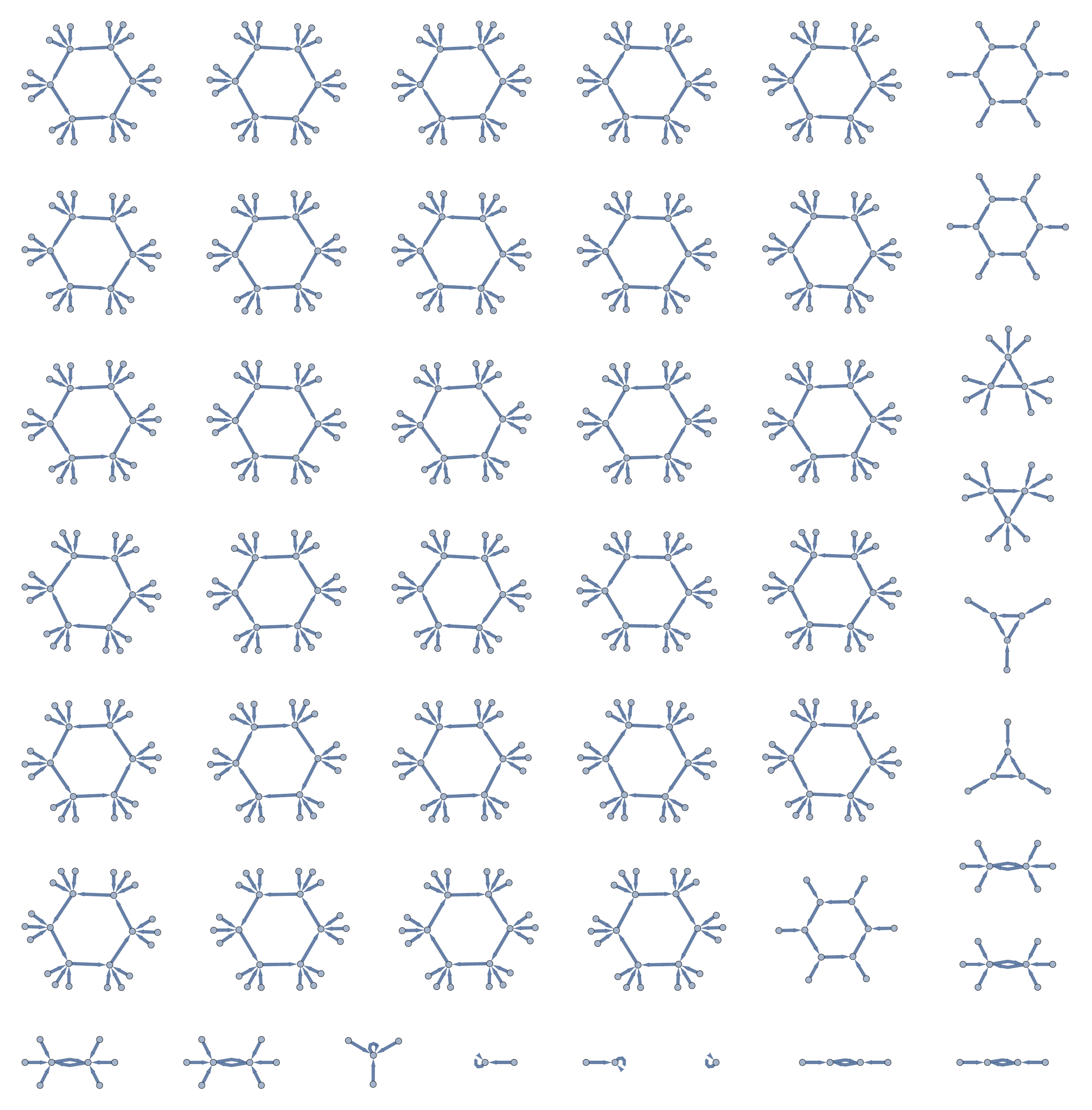}
\caption{The (unlabeled) graph $\gr{817}$, chosen because $817=19*43$ (q.v.~Figure~\ref{fig:18}) }
\label{fig:817}
\end{center}
\end{figure}

\subsection{When the nilpotents end up chillin' in the corner}

Consider $n=5^4=625$.  Here we want to study the graph $\gr{625}$, and in particular, focus on $\nilp{625}$.  

Let us use the formula in Theorem~\ref{thm:nilpotent}.  We first compute the graph $\tree 5 \ty 2$.  Recalling the graph $\units5$, we see that $1$ has two preimages: $1$ and $4$, and $4$ has two preimages: $2$ and $3$.  In this case the actual values of the preimages won't be important, since $\tree 5 \ty 1$ is a leaf regardless of $\ty$, due to the single power of $5$.  

According to the theorem, $\tree 5 1 2$ has five incoming copies of $\tree 5 1 1$ and five incoming copies of $\tree 5 4 1$, but these are all leaves, so $\tree 5 1 2$ is just a regular tree of width $10$, or $T_{10}^1$.  We get the same result for $\tree 5 4 2$.  And, of course, $\tree 5 2 2 = \tree 5 3 2$ is just a single node.

We also have that $\tree 5 \ty 3$ is a node simply because $3$ is odd.

Now, the in-neighborhood of $0$ will be any numbers with 2 or 3 powers of $5$ in them.  There are four such preimages with power three ($\tree  5 \ty 3$ for $\ty=1,2,3,4$, which are all nodes) and $\phi(25)=20$ such preimages of power two.  As we say above, half of these are trees $T_{10}^1$ and the other half are nodes.  Putting all of this together, the in-neighborhood of $0$ is $10$ copies of $T_{10}^1$ and $14$ leaves. See Figure~\ref{fig:625attractor}, where we have plotted $\nilp{625}$ .

\begin{figure}[th]
\begin{center}
\includegraphics[width=0.95\textwidth]{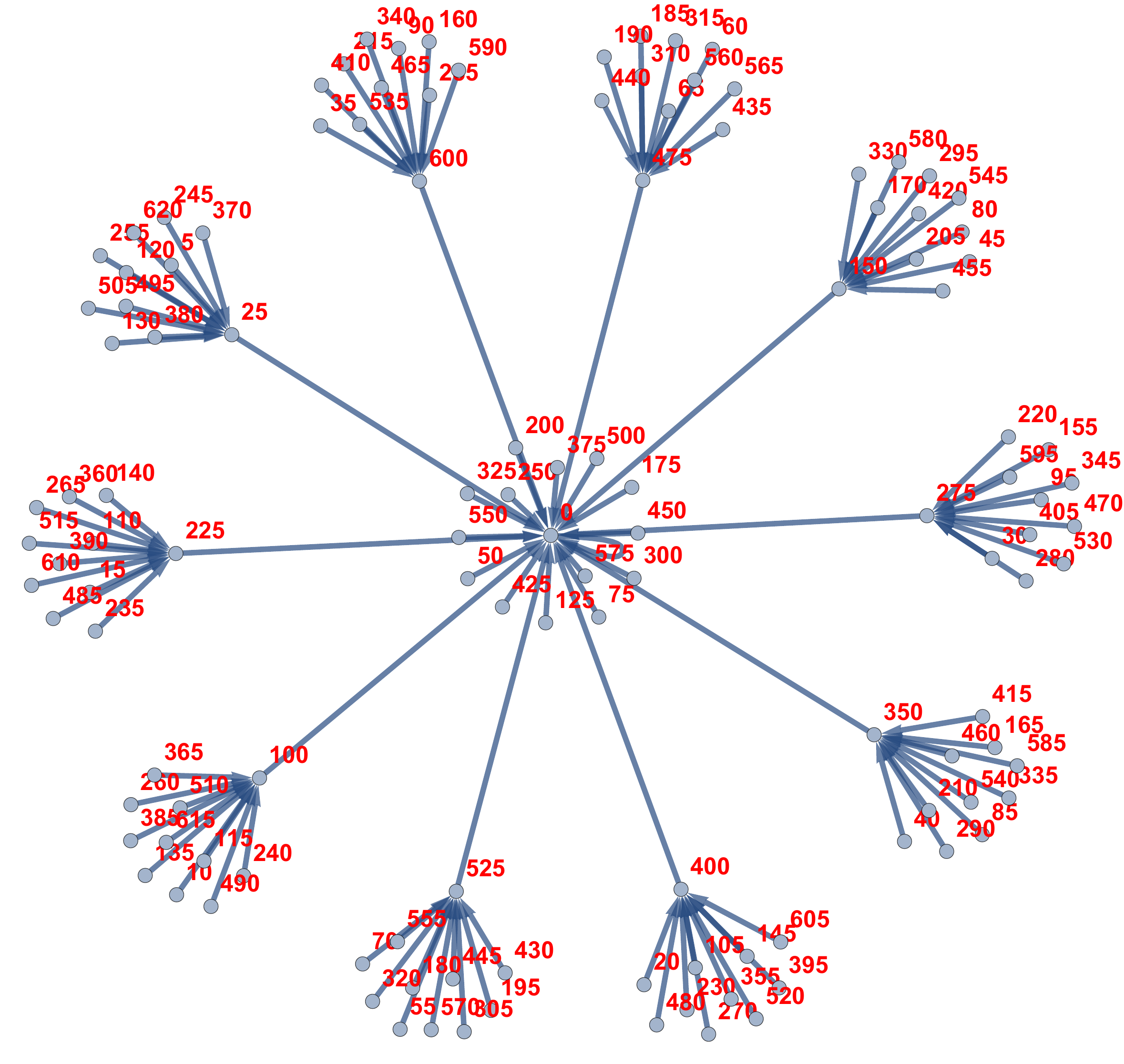}
\caption{The graph $\nilp{625}$.}
\label{fig:625attractor}
\end{center}
\end{figure}

Note the structure is as we say:  the node $0$ has fourteen leaves leading into it:  the four that come from $5^3$ terms:  125, 250, 375, 500, and the ten that come from $5^2$ terms, where the prefactor is $2$ or $3$ modulo 5:  $25*\{2,3,7,8,12,13,17,18,22,23\}$.  We also see that there are ten trees that go into $0$, coming from $5^2$ terms, where the prefactor is $1$ or $4$ modulo 5:  $25*\{1,4,6,9,11,14,16,19,21,24\}$.

Of course, if we want to understand the full graph $\gr{625}$ we also need to consider the units.  Note that $\phi(625) = 625-125 = 500 = 2^2*125$.  The divisors are $1,5,25,125$, where we have $\phi(125) = \ordt{125} = 100$, so one period-100 orbit, $\phi(25) = \ordt{25} = 20$, so one period-20 orbit, $\phi(5) = \ordt5=4$, so one period-4 orbit, and of course one fixed point.  Since $\theta=2$ these are all decorated with copies of $T_2^2$, so in fact we have 
\begin{equation*}
  \units{625} = C_{100}(T_2^2) \oplus C_{20}(T_2^2) \oplus C_{4}(T_2^2) \oplus T_2^2, 
\end{equation*}
so $\gr{625}$ has five components, see Figure~\ref{fig:625}.  (We can see that the tree $\nilp{625}$ is just chilling down there in the corner.)

\begin{figure}[th]
\begin{center}
\includegraphics[width=0.65\textwidth]{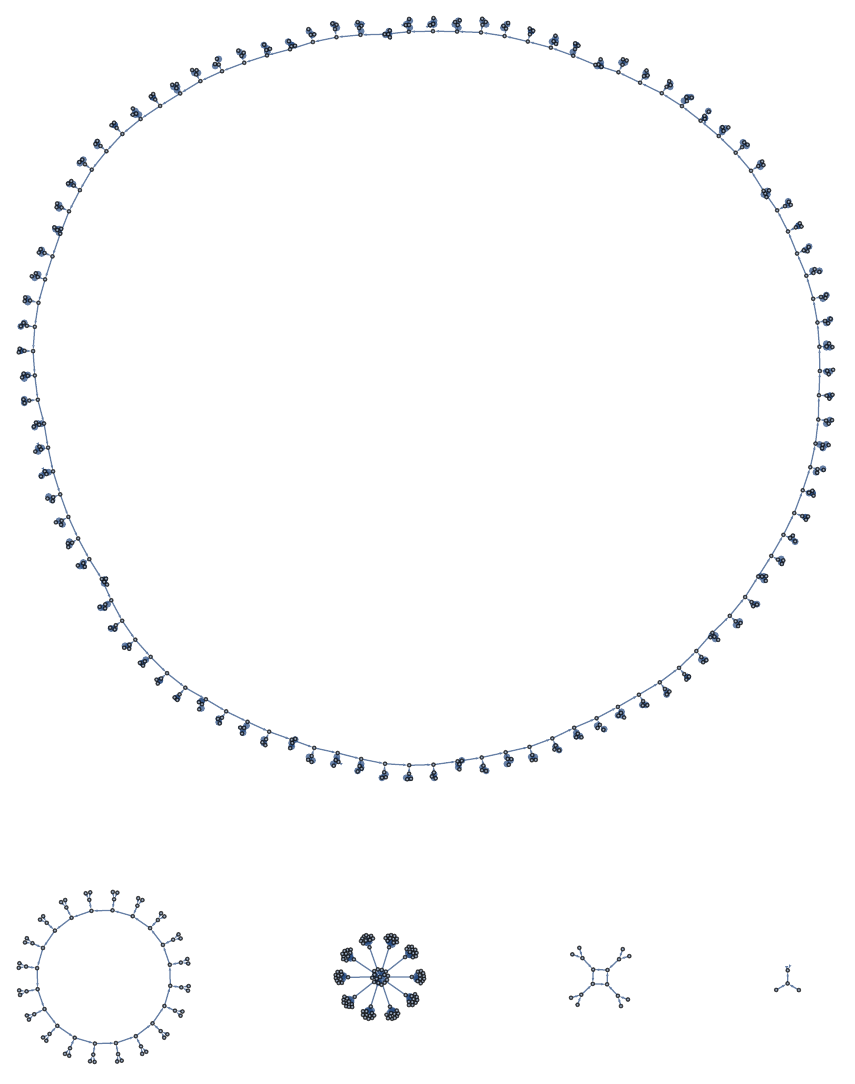}
\caption{The graph $\gr{625}$.}
\label{fig:625}
\end{center}
\end{figure}

\subsection{Well, *these* nilpotents ain't messing around}

Here we are going to describe the graph $\gr n$ when $n=177147=3^{11}$.

First we consider $\units{3^{11}}$.  We have $\phi(3^{11}) = 2\cdot 3^{10}$, so we have $\theta = 1$ and $\mu = 3^{10}$.  The divisors $d$ of $3^{10}$ are, of course, $3^k$ for $k=0,1,\dots, 10$.  In each of these cases, it turns out that $\phi(3^k) = \ordt{3^k} = 2\cdot 3^{k-1}$, so we have exactly 11 periodic orbits:  a fixed point, one of period 2, one of period 6, etc. all the way up to one of period $2\cdot 3^9 = 39366$.  Since $\theta=1$, all of these periodic orbits have a ``spoke'' coming out of each periodic point.

Now for $\nilp{3^{11}}$.  Let us first consider the neighborhood of $0$:  any $x$ of the form $x = \tx p^\ell$ with $\ell \ge 6$ will map directly into 0.  If $\ell$ is odd, then these are leaves.  If $\ell$ is even but $\tx \equiv 2\bmod 3$, then this is also a leaf.  In the other cases, we will have some more complex trees.  For each even $\ell$, if $\tx$ is $1\pmod 3$ we have a tree of type $\tree 3 1 \ell$, giving a total of   $1/2\phi(3^{11-k}) = 3^{10-k}$ of them, and if $\tx$ is $2\pmod 3$ it is a leaf, so we have  $1/2\phi(3^{11-k}) =3^{10-k}$ leaves.  If $\ell$ is odd we have $phi(3^{11-k}) = 2\cdot 3^{10-k}$ leaves.
Therefore the total number of leaves is
$1 + 6 + 9 + 54 + 81 = 151$.

From here we see that we should have attached to $0$:

\begin{table}[h]
\begin{center}
\begin{tabular}{|c|c|c|c|c|}\hline
Tree type &$\tree 3 1 {10}$ &$\tree 3 1 {8}$ &$\tree 3 1 {6}$ & leaves\\\hline
Count & 1 & 9 & 81 & 151\\\hline
\end{tabular}
\end{center}
\caption{The structures directly adjacent to $0$ in the graph $\nilp{3^{11}}$}
\label{tab:3^11}
\end{table}

Moreover, we can see that since $10=2\cdot 5$, $\tree 3 1 {10}$ is just a tree of depth one, since $y^2 = 3^{10}\bmod{3^{11}}$ iff $y = \tx 3^5$ with $\gcd(\tx,3)=1$ and $1\le \tx \le 3^6$, and there are $\phi(3^6) = 486$ such $\tx$'s, so $\tree 3 1 {10} \cong T_{486}^{(1)}$ is just a regular tree of width $486$ and depth $1$.  Similarly, $\tree 3 1 6$ is a tree of depth 1 and width $\phi(3^3) = 2\cdot 3^2 = 18$.

The more complicated case is $\tree 3 1 8$.  Note that $8=2^3$, so we can consider the unfurled graph of depth three $\U 3 1 3$, see Figure~\ref{fig:tree}.  According to Lemma~\ref{lem:layers}, we take this graph and expand it by the powers $3^4, 3^2, 3^1$, so that the in-neighborhood of $x=3^8$ has $3^4$ leaves coming in, and $3^4$ trees of type $\tree 3 1 4$, which themselves have $3^2$ leaves coming in, and $3^2$ trees coming in, and those are regular trees of the form $T_6^1$.  (And of course, for any $\tx 3^8$ with $\tx\equiv 1\bmod 3$ we get an isomorphic tree, so there are $9$ of them.)  See Figure~\ref{fig:3^8} for a visualization of $\tree 3 1 8$, which is also complicated and fun.  (We would have liked to put a visualization of the full $\nilp{3^{11}}$ here but at $59,049$ vertices and $177,147$ edges, it made our computer sad.)

\begin{figure}[th]
\begin{center}
\includegraphics[width=0.95\textwidth]{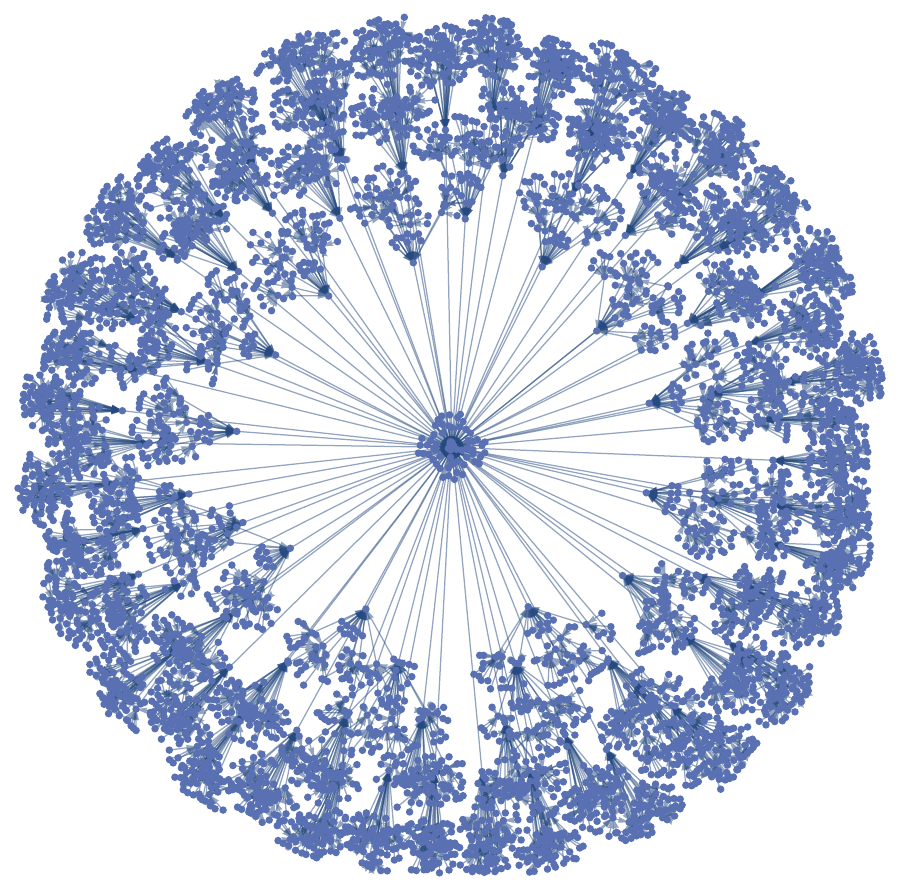}
\caption{The subgraph of $\gr{3^{11}}$ that goes to $x=3^8$, giving an example of $\tree 3 1 8$.  Its details are given in the text, but note that there are {\bf nine} copies of this bad boy attached to zero, not to mention all of the other things going in (see Table~\ref{tab:3^11}).}
\label{fig:3^8}
\end{center}
\end{figure}

\section{Conclusions}

We have presented some basic theory and give a few nice examples.  The examples were a lot of fun, and there's no doubt that one can find a whole host of other interesting examples.  I leave it to the readers of this article to explore and find more.

\end{document}